\documentclass[reqno]{amsart}
\usepackage{a4wide}
\usepackage{fixltx2e}
\usepackage{type1cm}
\usepackage{tikz,microtype}
\usetikzlibrary{calc}
\usepackage[ruled,noend]{algorithm2e}

\newcommand{\tens}[1]{\langle #1 \rangle}
\newcommand{\cc}{\mathbb C}
\newcommand{\rr}{\mathbb R}
\newcommand{\cO}{\mathcal O}

\newcommand{\Sym}{\mathrm{Sym}}

\newtheorem*{theorem}{Theorem}
\newtheorem*{fact}{Fact}

\newtheorem{lemma}{Lemma}

\theoremstyle{definition}
\newtheorem*{defi}{Definition}

\begin{document}

\title{Upgrading Subgroup Triple Product Property Triples}
\author{Ivo Hedtke}

\address{Institute of Computer Science, University of Halle-Wittenberg, D-06099 Halle, Germany}
\email{hedtke@informatik.uni-halle.de}

\begin{abstract}
In 2003 \textsc{Cohn} and \textsc{Umans} introduced a group-theoretic approach to fast matrix
multiplication. This involves finding large subsets of a group $G$ satisfying the Triple Product
Property (TPP) as a means to bound the exponent $\omega$ of matrix multiplication. Recently, \textsc{Hedtke} and \textsc{Murthy} discussed several methods to find TPP triples. Because the search space for subset triples is too large, it is only possible to focus on subgroup triples.

We present methods to upgrade a given TPP triple to a bigger TPP triple. If no upgrade is possible we use reduction methods (based on random experiments and heuristics) to create a smaller TPP triple that can be used as input for the upgrade methods.

If we apply the upgrade process for subset triples after one step with the upgrade method for subgroup triples we achieve an enlargement of the triple size of 100 \% in the best case.
\end{abstract}

\maketitle

\section{Introduction}\thispagestyle{empty}

\subsection{A Very Short History of Fast Matrix Multiplication}
The naive algorithm for matrix multiplication is an $\cO(n^3)$ algorithm.
From \textsc{Strassen} (see \cite{Strassen1969}) we know that there is an $\cO(n^{2.81})$ algorithm for this problem.
The fastest
known algorithm runs in $\cO(n^{2.376})$ time (see \cite{Coppersmith1987} from \textsc{Coppersmith} and \textsc{Winograd}). Most researchers believe that an optimal
algorithm with $\cO(n^2)$ runtime exists, but since 1987 no further
progress was made in finding one.

In this paper we only focus on the complexity of fast matrix multiplication, that means on nontrivial bounds for the exponent
$
\omega := \inf \{r \in \rr  :  M(n)=\cO(n^r)\}
$
of matrix multiplication. Here $M(n)$ denotes the number of field operations in characteristic $0$ required to multiply two $(n\times n)$ matrices. Details about the complexity of matrix multiplication and the exponent $\omega$ can be found in~\cite{Buergisser1997}.

\subsection{The Group-Theoretic Approach from Cohn and Umans}
In 2003 \textsc{Cohn} and \textsc{Umans} introduced in \cite{Cohn2003} a group-theoretic approach to fast matrix multiplication. In 2005, together with \textsc{Kleinberg} and \textsc{Szededy}, they could achieve the upper bound $2.41$ for $\omega$ (see \cite{Cohn2005}).
(The reader can find the necessary background on group- and representation theory in \cite{Alperin1991} and \cite{James2001}.)

The main idea is to embed the matrix multiplication over a ring $R$ into the group ring $RG$, where $G$ is a (finite) group. A group $G$ admits such an embedding, if there are subsets $S$, $T$ and $U$ which fulfill the so-called \emph{Triple Product Property}. 

\begin{defi}[right quotient]
Let $G$ be a group and $\emptyset\neq X \subseteq G$ be a nonempty subset of $G$. The \emph{right quotient} $Q(X)$ of $X$ is defined by $Q(X):=\{xy^{-1} : x,y \in X\}$.
\end{defi}

Note that $Q(S)=S$ holds iff $S$ is a subgroup of $G$.

\begin{defi}[Triple Product Property]
We say that the nonempty subsets $S$, $T$ and $U$ of a group $G$ fulfill the \emph{Triple Product Property} (TPP) if for $s\in Q(S)$, $t\in Q(T)$ and $u \in Q(U)$,
$stu=1$ holds iff $s=t=u=1$.
\end{defi}

With $\tens{n,p,m}$ we denote the problem
$
\tens{n,p,m} \colon \cc^{n\times p}\times \cc^{p\times m} \to \cc^{n\times m}$, $(A,B)\mapsto AB
$
of multiplying an $(n\times p)$ with a $(p\times m)$ matrix over $\cc$. We say that a group $G$ \emph{realizes} $\tens{s_1,s_2,s_3}$ if there are subsets $S_i\subseteq G$ of sizes $|S_i|=s_i$, which fulfill the TPP. In this case we call $(S_1,S_2,S_3)$ a \emph{TPP triple} of $G$.

Let us now focus on the embedding of the matrix multiplication into $\cc G$. Let $G$ realize $\tens{n,p,m}$ through the subsets $|S|=n$, $|T|=p$ and $|U|=m$. Let $A$ be an $(n\times p)$ and $B$ be a $(p\times m)$ matrix. We index the entries of $A$ and $B$ with the elements of $S$, $T$ and $U$ instead of numbers. Now we have
\[
(AB)_{s,u}=\sum\nolimits_{t \in T} A_{s,t}B_{t,u}.
\]
\textsc{Cohn} and \textsc{Umans} showed that this is the same as the coefficient of $s^{-1}u$ in the product
\begin{gather*}
\Big( \sum\nolimits_{s\in S, t\in T} A_{s,t}s^{-1}t\Big)
\Big( \sum\nolimits_{\hat t\in T,u\in U} B_{\hat t,u}\hat t^{-1}u\Big).
\end{gather*}
So we can read off the matrix product from the group ring product by looking at the coefficients of $s^{-1}u$ with $s\in S$ and $u\in U$.

\begin{defi}[TPP (subgroup) capacity]\label{def:TPPcap}
We define the \emph{TPP capacity} $\beta (G)$ of a group $G$ as
\[
\beta(G) := \max\{npm : G\text{ realizes }\tens{n,p,m}\}
\] and the \emph{TPP subgroup capacity} $\beta_\mathrm{g}(G)$ of $G$ as 
\[
\beta_\mathrm{g}(G) := \max\{npm : G\text{ realizes }\tens{n,p,m}\text{ through subgroups}\}.
\]
\end{defi}

Note that 
$
\beta (G) \geq \beta_\mathrm{g}(G) \geq  |G|
$
holds, because every group $G$ realizes $\tens{|G|,1,1}$ through the TPP triple $(G,1,1)$, and the search space for $\beta$ includes the one for $\beta_\mathrm{g}$.

\begin{defi}[$r$-character capacity]\label{def:DG}
Let $G$ be a group with the character degrees $\{d_i\}$. We define the \emph{$r$-character capacity} of $G$ as $D_r(G):=\sum_i d_i^r$.
\end{defi}

We can now use $\beta$ and $D_r$ to get new bounds for $\omega$:

\begin{theorem}\textup{\cite[Thm. 4.1]{Cohn2003}}
If $G\neq 1$ is a finite group, then $\beta(G)^{\omega/3} \leq D_\omega (G)$.
\end{theorem}

Note, that the inequality above yields to a nontrivial upper bound for $\omega$ iff $\beta(G) > D_3(G)$.

\subsection{The Aim of this Work}
With the  theorem above it is possible to find (new) nontrivial bounds for $\omega$. Therefore, we want to compute $\beta(G)$ for as many groups as possible (for example to find a counterexample of a group $G$ that realizes a nontrivial upper bound for $\omega$, or a \emph{new} nontrivial bound for $\omega$). One method is a brute-force search for TPP triples in a given group. But the size of the search space is too large: $\mathcal O(4^{|G|})$. Note that the currently best search algorithm for subset TPP triples has a worst case runtime of $\mathcal O(8^{|G|})$, but it is faster (based on some heuristics) than a full search in the search space of size $\mathcal O(4^{|G|})$. Thus we only know efficient search methods for subgroup TPP triples (described in \cite{HedtkeMurthy2011}). But in many cases, $\beta_\mathrm{g}(G) < \beta (G)$ hold. An example is the dihedral group of order $10$ with $\beta_\mathrm{g}(D_{10})=10$ and $\beta(D_{10})=12$ (see \cite{HedtkeMurthy2011}). Therefore, we use the known search methods only to find subgroup TPP triples and after this we try to upgrade these triples to subset TPP triples of a bigger size. This will give us a better lower bound $\ell$ for the TPP capacity with $\beta_\mathrm{g}(G) \leq \ell(G) \leq \beta(G)$. If we were able to upgrade a given maximal subgroup TPP triple to a bigger subset TPP triple we can use $\ell(G)^{\omega/3} \leq D_\omega (G)$ instead of $\beta_\mathrm{g}(G)^{\omega/3} \leq D_\omega (G)$ to find a better upper bound for $\omega$.

In this paper we present methods to upgrade a given subgroup (or subset) TPP triple (for example found by a brute-force search) to a bigger subset TPP triple. If no upgrade is possible we use a reduction method to create a smaller TPP triple that can be used as input for the upgrade methods. The ways of how to shrink a given TPP triple are based on random experiments and heuristics.

\subsection{Fundamentals} Now we collect some facts about TPP triples which we will use later.
In the whole paper we only consider \emph{finite nonabelian groups}, because we want to work with $|G|$ and \textsc{Cohn} and \textsc{Umans} proved that $\beta(G)=|G|=D_2(G)\leq D_3(G)$ if $G$ is abelian (see \cite[Lem. 3.1]{Cohn2003}).

\begin{fact}\textup{\cite[Thm. 5]{Hedtke2011}}
If $(S,T,U)$ is a TPP triple in $G$, then
\begin{gather}\label{eq:Q}
|Q(S)| + |Q(T)| + |Q(U)| \leq |G| + 2.
\end{gather}
\end{fact}

\begin{fact}\textup{\cite[Obs. 3.1]{Neumann2011}}
If $(S,T,U)$ is a TPP triple in $G$, then 
\begin{gather}\label{eq:Neumann}
|S|(|T|+|U|-1)\leq |G|, \quad |T|(|S|+|U|-1)\leq |G| \quad \text{and} \quad |U|(|S|+|T|-1)\leq |G|.
\end{gather}
\end{fact}

\begin{fact}\textup{\cite[Lem. 2.1]{Cohn2003}}
If $G$ realizes $\tens{n,p,m}$, then it does so for every permutation of $n$, $p$ and $m$.
\end{fact}

\begin{fact}\textup{\cite[Obs. 2.5]{HedtkeMurthy2011}}
It is sufficient to search for triples with
\begin{gather}\label{eq:Order}
|S| \geq |T| \geq |U|.
\end{gather}
\end{fact}

Note, that in the following text we always assume that \eqref{eq:Order} holds.

\begin{defi}[size]
We define
$
|(S,T,U)|=|S| \cdot |T| \cdot |U|
$
as the \emph{size} of a triple $(S,T,U)$.
\end{defi}

\begin{defi}[basic]
According to \textsc{Neumann} we call a TPP triple $(S,T,U)$ that fulfills $1\in S\cap T\cap U$ a \emph{basic} TPP triple.
\end{defi}

Note that for a basic TPP triple herefrom $1 = S\cap T = S\cap U = T\cap U$ follows by \cite[Thm.~1]{Hedtke2011}. The following fact is an equivalent to the TPP definition. It is usefull for TPP tests in the following sections.

\begin{fact}\textup{\cite[Thm. 3.1]{HedtkeMurthy2011}}
Three subsets of $G$ form a basic TPP triple $(S,T,U)$ iff
\begin{gather}\label{eq:HedtkeMurthy}
1 \in S\cap T \cap U,\qquad
Q(T) \cap Q(U) = 1\qquad\text{and}\qquad
Q(S) \cap Q(T)Q(U) = 1.
\end{gather}
\end{fact}

Note that the statement from the fact above holds for any permutation of $S$, $T$ and $U$. That means, $(S_1,S_2,S_3)$ is a basic TPP triple iff
\begin{gather}\label{fact:Sym}
1\in S_1 \cap S_2 \cap S_3,  ~
Q(S_{\pi(2)}) \cap Q(S_{\pi(3)}) = 1, ~
Q(S_{\pi(1)}) \cap Q(S_{\pi(2)})Q(S_{\pi(3)})=1 \quad \forall \pi \in \Sym(3).
\end{gather}

Depending on the context, the symbol $1$ will denote either the number $1$, the group identity $1_G$, or the trivial subgroup $\{1_G\}$.

\subsection{Overview of the Upgrade Process}
The diagram in Figure~\ref{fig:diag} gives us an overview of the different upgrade steps and methods described in this paper. The details can be found in the corresponding (sub)sections.

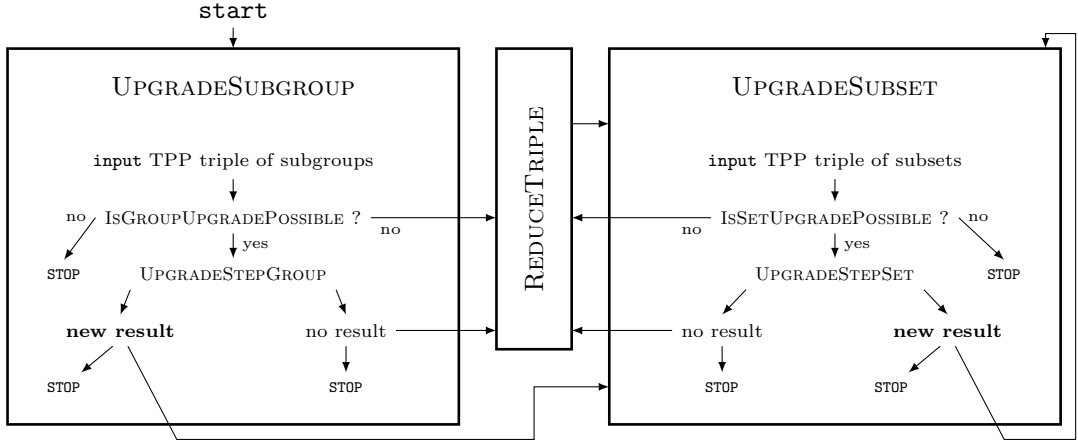
\begin{figure}
\centering
\begin{tikzpicture}
  \draw (3,.5) node (start) {\texttt{start}};
  \draw[-latex] (start) -- (3,0);
  \draw[line width=1pt] (0,0) rectangle (6,-5);
  \draw[line width=1pt] (6.5,0) rectangle node[rotate=90] {\textsc{ReduceTriple}} (7.5,-4);
  \draw[line width=1pt] (8,0) rectangle (14,-5);

  \draw     (3,-0.5)          node        {\textsc{UpgradeSubgroup}};
  \draw     (3,-1.5)          node (grp1) {\scriptsize \texttt{input} TPP triple of subgroups};
  \draw     (3,-2.25)         node (grp2) {\scriptsize \textsc{IsGroupUpgradePossible} ?};
  \draw[-latex] (grp1)                        -- (grp2);
  \draw[-latex] (grp2.east)       node        [below right] {\tiny no} -- (6.5,-2.25);
  \draw     (3,-2.625)        node        [right] {\tiny yes};
  \draw     (0.75,-3)         node (s1)   {\tiny \texttt{STOP}};
  \draw[-latex] (grp2.west)       node        [left] {\tiny no} -- (s1.north);
  \draw     (3,-3)            node (grp3) {\scriptsize \textsc{UpgradeStepGroup}};
  \draw[-latex] (grp2)                        -- (grp3);
  \draw     (1.5,-3.75)       node (res1) {\scriptsize\bfseries new result};
  \draw     (4.5,-3.75)       node (res2) {\scriptsize no result};
  \draw[-latex] (grp3.south west)             -- (res1.north);
  \draw[-latex] (grp3.south east)             -- (res2.north);
  \draw     (0.75,-4.5)       node (st1)  {\tiny \texttt{STOP}};
  \draw     (4.5,-4.5)        node (st2)  {\tiny \texttt{STOP}};
  \draw[-latex] ($(res1.south) - (.1,0)$)             -- (st1);
  \draw[-latex] (res2)                        -- (st2);
  \draw[-latex] (res2)                        -- (6.5,-3.75);
  \draw[-latex] ($(res1.south) + (.1,0)$)             -- (2.25,-5.2) -- (7,-5.2) -- (7,-4.5) -- (8,-4.5);

  \draw[-latex] (7.5,-1) -- (8,-1);

  \draw     (3+8,-0.5)          node        {\textsc{UpgradeSubset}};
  \draw     (3+8,-1.5)          node (lgrp1) {\scriptsize \texttt{input} TPP triple of subsets};
  \draw     (3+8,-2.25)         node (lgrp2) {\scriptsize \textsc{IsSetUpgradePossible} ?};
  \draw[-latex] (lgrp1)                        -- (lgrp2);
  \draw[-latex] (lgrp2.west)       node        [below left] {\tiny no} -- (7.5,-2.25);
  \draw     (3+8,-2.625)        node        [right] {\tiny yes};
  \draw     (13.25,-3)         node (ls1)   {\tiny \texttt{STOP}};
  \draw[-latex] (lgrp2.east)       node        [right] {\tiny no} -- (ls1.north);
  \draw     (3+8,-3)            node (lgrp3) {\scriptsize \textsc{UpgradeStepSet}};
  \draw[-latex] (lgrp2)                        -- (lgrp3);
  \draw     (1.5+8,-3.75)       node (lres1) {\scriptsize no result};
  \draw     (4.5+8,-3.75)       node (lres2) {\scriptsize\bfseries new result};
  \draw[-latex] (lgrp3.south west)             -- (lres1.north);
  \draw[-latex] (lgrp3.south east)             -- (lres2.north);
  \draw     (11.75,-4.5)       node (lst1)  {\tiny \texttt{STOP}};
  \draw     (9.5,-4.5)        node (lst2)  {\tiny \texttt{STOP}};
  \draw[-latex] (lres1) -- (7.5,-3.75);
  \draw[-latex] (lres1) -- (lst2);
  \draw[-latex] ($(lres2.south) - (.1,0)$) -- (lst1);
  \draw[-latex] ($(lres2.south) + (.1,0)$) -- (13.25,-5.2) -- (14.2,-5.2) -- (14.2,0.2) -- (13.8,.2) -- (13.8,0);
\end{tikzpicture}
\caption{Overview of the different upgrade steps and methods.}
\label{fig:diag}
\end{figure}

\section{Upgrading of TPP Subgroup Triples}

\noindent Consider a given \emph{maximal} TPP triple $(S,T,U)$ of subgroups of a group $G$ (for example found by a brute-force search), which means that there is no other TPP triple of subgroups in $G$ with a bigger size. Because of $\beta_\mathrm{g}(G) \leq \beta(G)$ it is possible, that there is a TPP triple $(\tilde S, \tilde T, \tilde U)$ of subsets of $G$ with a bigger size
\[
|(S,T,U)| < |(\tilde S,\tilde T,\tilde U)|.
\]
In this section we describe a method that tries to find such a bigger TPP triple by upgrading the given triple $(S,T,U)$, i.\,e. $S \subseteq \tilde S$, $T \subseteq \tilde T$ and $U\subseteq \tilde U$.

\subsection{The upgrade step} If an upgrade of a given triple $(S,T,U)$ is possible (see the next subsection) we proceed as follows: We compute the set
\[
C:=C_{S,T,U}^G := G \setminus(S \cup T\cup U)
\]
of all possible candidates for such an upgrade. If $C\neq \emptyset$ we pick a $c\in C$ and check, if with
\[
\tilde S := S \cup \{c\}, \qquad 
\tilde T := T \cup \{c\} \qquad \text{or} \qquad 
\tilde U := U \cup \{c\}
\]
one of the following triples fulfills the TPP too:
\[
(\tilde S,T,U), \qquad 
(S, \tilde T, U) \qquad \text{or} \qquad 
(S, T, \tilde U).
\]
If this is the case, we start the upgrade step again for one of the triples above, but this time the upgrade method for subset triples. If none of the triples fulfill the TPP, we update $C:=C\setminus\{c\}$ and try it again. If $C=\emptyset$ there is no possibility to upgrade the given triple.

\subsection{Is an upgrade possible?} We don't give a complete answer to this question (that means we will not answer the question whether the upgrade process will give us a bigger TPP triple), but we will indicate if the upgrade process could yield to a new subset TPP triple, which leads to a speedup of the process.

First we focus on the upgrade itself. Let $X$ be a subset (or subgroup) of $G$ and $c\in G$ not be in $X$. Then we can compute $\tilde X := X\cup \{c\}$ like
\begin{gather}\label{eq:upgrade}
Q(\tilde X) = Q(X \cup \{c\}) = Q(X) \cup cX^{-1} \cup Xc^{-1} \stackrel{X \text{ is group}}{=}
X \cup cX \cup Xc^{-1}.
\end{gather}

\begin{lemma}\label{lemm:SchrankenUpdate}
If $X$ is a subgroup of $G$ and $c\in G \setminus X$, then
\[
2|X| \leq |Q(X \cup \{c\})| \leq 3|X|.
\]
\end{lemma}
\begin{proof}
The right-hand-side of the statement follows directly from \eqref{eq:upgrade}. We prove the left-hand-side by showing that $X\cap cX = \emptyset$. Assume that there is a common element $y\in X\cap cX$. In this case there is a $x\in X$ with $y=cx$, and so $c=yx^{-1} \in X$ because $X$ is a group, a contradiction.
\end{proof}

To check if the upgrade process is possible, we compute the set $I(S,T,U) \subseteq \{S, T, U\}$ of the possible places for an upgrade. If $I(S,T,U)=\emptyset$ no upgrade of $(S,T,U)$ is possible. In the other case we start the upgrade procedure. Let $\xi:=|G| + 2 - |S| - |T| - |U|$. With this we can write \eqref{eq:Q} as $0 \leq \xi$. Now we define $I$ by combining \eqref{eq:Q}, \eqref{eq:Neumann} and Lemma~\ref{lemm:SchrankenUpdate} as follows:
\begin{align*}
S\in I(S,T,U) \Leftrightarrow &
\begin{cases}
|S| \leq \xi\qquad\qquad\qquad\qquad\qquad & (\ref{eq:Q}_\mathrm{g}^S)\\
\text{and}\\
(|S| + 1) ( |T| + |U|  - 1) \leq |G| & (\ref{eq:Neumann}^S)
\end{cases}\\
T\in I(S,T,U) \Leftrightarrow &
\begin{cases}
|T| \leq \xi\qquad\qquad\qquad\qquad\qquad & (\ref{eq:Q}_\mathrm{g}^T)\\
\text{and}\\
|S| ( |T| + |U|) \leq |G| & (\ref{eq:Neumann}^T)
\end{cases}\\
U\in I(S,T,U) \Leftrightarrow &
\begin{cases}
|U| \leq \xi\qquad\qquad\qquad\qquad\qquad & (\ref{eq:Q}_\mathrm{g}^U)\\
\text{and}\\
|S| ( |T| + |U|) \leq |G| & (\ref{eq:Neumann}^U)
\end{cases}
\end{align*}

Remember that we assumed that $|S|\geq |T| \geq |U|$. Furthermore, note that $(\ref{eq:Neumann}^T) = (\ref{eq:Neumann}^U)$. The following algorithm computes $I(S,T,U)$:

\begin{algorithm}[H]
\DontPrintSemicolon
$I:=\emptyset$\;
\lIf{$(\ref{eq:Neumann}^S)$ \textbf{\textup{and}} $(\ref{eq:Q}_\mathrm{g}^S)$}{$I:=I\cup\{S\}$\;}
\If{$(\ref{eq:Neumann}^T)$}{
\lIf{$(\ref{eq:Q}_\mathrm{g}^T)$}{$I:=I\cup\{T\}$\;}
\lIf{$(\ref{eq:Q}_\mathrm{g}^U)$}{$I:=I\cup\{U\}$\;}
}
\Return{$I$\;}
\caption{\textsc{IsGroupUpgradePossible}$(S,T,U)$}
\end{algorithm}

\subsection{Checking all possibilities} Finally we check all possible upgrades $C$ in $I(S,T,U)$ to compute the set $\mathcal U(S,T,U)$ of all upgrades of the given subgroup TPP triple $(S,T,U)$. We do this by explicit calculations of $Q(\cdot)$ in a specialized algorithm.
If $\mathcal U=\emptyset$ there is no upgrade of the given triple. In this case we can \emph{stop} the process or we use the \emph{reduction method} described later to reduce the given triple to a smaller \emph{subset} TPP triple as an input for the \emph{upgrade process for subsets}, which we also disuss later.
If $\mathcal U \neq \emptyset$ there are upgrades and we can compute the new \emph{lower bound} $\ell(G)$ for $\beta(G)$. Maybe we can get a better upper bound for $\omega$ with it. Or we use the upgraded TPP triple (which is no subgroup TPP triple any more!) as input for the upgrade process for subsets. In this case it is possible, that we get an even better bound $\ell' > \ell$.

\begin{algorithm}[H]
\DontPrintSemicolon
$\mathcal U:=\emptyset$; ~~~ $I:={}$\textsc{IsGroupUpgradePossible}$(S,T,U)$; ~~~ $C:=C_{S,T,U}^G$\;
\If{$S\in I$}{
$P:=TU$\;
\ForEach{$c\in C$}
{
\lIf{\textsc{SpecialTPPTestGroup}$(S,c,P)$}{$\mathcal U := \mathcal U \cup \{(S\cup \{c\}, ~ T, ~ U)\}$\;}
}
}
\If{$T\in I$}{
$P:=SU$\;
\ForEach{$c\in C$}
{
\lIf{\textsc{SpecialTPPTestGroup}$(T,c,P)$}{$\mathcal U := \mathcal U \cup \{(S, ~ T\cup \{c\}, ~ U)\}$\;}
}
}
\If{$U\in I$}{
$P:=ST$\;
\ForEach{$c\in C$}
{
\lIf{\textsc{SpecialTPPTestGroup}$(U,c,P)$}{$\mathcal U := \mathcal U \cup \{(S, ~ T, ~ U\cup \{c\})\}$\;}
}
}
\Return{$\mathcal U$\;}
\caption{\textsc{UpgradeStepGroup}$(S,T,U)$}
\end{algorithm}

The algorithm above\enlargethispage{\baselineskip} computes the set of \emph{all} upgrades. This could need a lot of time and space. Therefore, one can consider another version that only returns the biggest upgrade (that means an upgraded TPP triple $T'$ with $|T|\leq |T'|$ for all $T\in \mathcal U$). Hence we change the order of the tests for the members in $I$ (because we assume that \eqref{eq:Order} holds) and return the first result we find:

\begin{algorithm}[H]
\DontPrintSemicolon
$I:={}$\textsc{IsGroupUpgradePossible}$(S,T,U)$; ~~~ $C:=C_{S,T,U}^G$\;
\If{$U\in I$}{
$P:=ST$\;
\ForEach{$c\in C$}
{
\lIf{\textsc{SpecialTPPTestGroup}$(U,c,P)$}{\Return{$(S, ~ T, ~ U\cup \{c\})$\;}}
}
}
\If{$T\in I$}{
$P:=SU$\;
\ForEach{$c\in C$}
{
\lIf{\textsc{SpecialTPPTestGroup}$(T,c,P)$}{\Return{$(S, ~ T\cup \{c\}, ~ U)$\;}}
}
}
\If{$S\in I$}{
$P:=TU$\;
\ForEach{$c\in C$}
{
\lIf{\textsc{SpecialTPPTestGroup}$(S,c,P)$}{\Return{$(S\cup \{c\}, ~ T, ~ U)$\;}}
}
}

\Return{$\emptyset\;$}
\caption{\textsc{UpgradeStepGroupBiggest}$(S,T,U)$}
\end{algorithm}

We will use \textsc{UpgradeStepGroup} only, if \textsc{UpgradeStepGroupBiggest} found an upgrade.

\subsection{Special TPP tests} As we are in the process of upgrading subgroups, there are easy tests for the TPP. We use the equivalent to the TPP by \textsc{Hedtke}\&\textsc{Murthy} (see \eqref{eq:HedtkeMurthy}).

If we upgrade $S$ to $\tilde S$, we only need to check $Q(\tilde S)\cap TU=1$, because $T$ and $U$ are subgroups and we don't upgrade them and so the other equations in \eqref{eq:HedtkeMurthy} are fulfilled. We focus on this in detail. We have
\begin{align*}
Q(\tilde S)\cap TU= (S \cup cS \cup Sc^{-1}) \cap TU
= \underbrace{(S \cap TU)}_{=1} \cup \underbrace{(cS \cap TU)}_{=:A} \cup \underbrace{(Sc^{-1} \cap TU)}_{=:B}.
\end{align*}
It follows that $Q(\tilde S)\cap TU=1$ iff $A=\emptyset$ and $B=\emptyset$. ($A=1$ would also be possible, but in that case it follows that $c\in S$, a contradiction. The same holds for $B$.) 

From \eqref{fact:Sym} we know, that the TPP does not depend on the order of $S$, $T$ and $U$.
Thus we only have to check $cT\cap SU = Tc^{-1}\cap SU=\emptyset$ if we upgrade $T$. The same holds for $U$.

\begin{algorithm}[H]
\DontPrintSemicolon
\eIf{$cX \cap P=\emptyset$ \textbf{\textup{and}} $Xc^{-1} \cap P=\emptyset$}{\Return{\texttt{true}}\;}{\Return{\texttt{false}}\;}
\caption{\textsc{SpecialTPPTestGroup}$(X,c,P)$}
\end{algorithm}

\section{Upgrading of TPP Subset Triples}

\noindent In the case that the given initial basic TPP triple $(S,T,U)$ is a triple of subsets, where at least one of $S$, $T$ or $U$ is no group, we have to change the upgrade procedure. It would be possible to use the method developed in this section also for subgroup TPP triples, but the method for subgroups is faster.

Without loss of generality we assume, that in the given TPP triple $(S,T,U)$ all of $S$, $T$ and $U$ are subsets. In this case we define the set of upgrade candidates of $S$ as
\[
C(S):=C_{S,T,U}^G(S) := G \setminus (S \cup Q(T) \cup Q(U)),
\]
and $C(T)$ and $C(U)$ in the same way. Again we first focus on the update of $Q(\cdot)$. Unfortunately, it is not possible to prove such a strong result as in Lemma~\ref{lemm:SchrankenUpdate}. We only know the trivial fact that:

\begin{lemma}\label{lemm:SchrankenTeilmengen}
If $X$ is a subset of $G$ such that $1\in X$ and $c\in G \setminus X$, then
\[
|Q(X)| \leq |Q(X \cup \{c\})| \leq |Q(X)| + 2|X|.\qedhere
\]
\end{lemma}

Because of the lemma above it is not possible to give an analogon to $(\ref{eq:Q}_\mathrm{g}^S)$. Again we want do indicate whether an upgrade could be successful. We define $I(S,T,U)$ as follows:
\begin{align*}
S\in I(S,T,U) \Leftrightarrow (\ref{eq:Neumann}^S)\\
T\in I(S,T,U) \Leftrightarrow (\ref{eq:Neumann}^T)\\
U\in I(S,T,U) \Leftrightarrow (\ref{eq:Neumann}^U)
\end{align*}

The new algorithm for subsets is somewhat shorter but not as strong as the one for subgroups:

\begin{algorithm}[H]
\DontPrintSemicolon
$I:=\emptyset$\;
\lIf{$(\ref{eq:Neumann}^S)$}{$I:=I\cup\{S\}$\;}
\lIf{$(\ref{eq:Neumann}^T)$}{$I:=I\cup\{T,U\}$\;}
\Return{$I$\;}
\caption{\textsc{IsSetUpgradePossible}$(S,T,U)$}
\end{algorithm}

Let us assume, that we upgrade $S$. The only thing we have to check is
\begin{align*}
1&=Q(\tilde S) \cap Q(T)Q(U)= (Q(S) \cup cS^{-1} \cup Sc^{-1})\cap Q(T)Q(U)\\
&= \underbrace{(Q(S) \cap Q(T)Q(U))}_{=1} \cup \underbrace{(cS^{-1} \cap Q(T)Q(U))}_{=:A} \cup \underbrace{(Sc^{-1} \cap Q(T)Q(U))}_{=:B}.
\end{align*}
If we assume that $A=1$, then $c\in S$, a contradiction. Therefore $A=\emptyset$ must hold. $B=\emptyset$ holds, too. Therefore the resulting TPP test is similar to \textsc{SpecialTPPTestGroup}, but in our current situation we use $P:=Q(T)Q(U)$ and take care of $X^{-1}$:

\begin{algorithm}[H]
\DontPrintSemicolon
\eIf{$cX^{-1} \cap P=\emptyset$ \textbf{\textup{and}} $Xc^{-1} \cap P=\emptyset$}{\Return{\texttt{true}}\;}{\Return{\texttt{false}}\;}
\caption{\textsc{SpecialTPPTestSet}$(X,c,P)$}
\end{algorithm}

The upgrade procedure \textsc{UpgradeStepSet} differs from \textsc{UpgradeStepGroup} in the definitions of $P$, the candidates $C(\cdot)$ and the TPP test:

\begin{algorithm}[H]
\DontPrintSemicolon
$\mathcal U:=\emptyset$;  $I:={}$\textsc{IsSetUpgradePossible}$(S,T,U)$\;
\If{$S\in I$}{
$P:=Q(T)Q(U)$\;
\ForEach{$c\in C(S)$}
{
\lIf{\textsc{SpecialTPPTestSet}$(S,c,P)$}{$\mathcal U := \mathcal U \cup \{(S\cup \{c\}, ~ T, ~ U)\}$\;}
}
}
\If{$T\in I$}{
$P:=Q(S)Q(U)$\;
\ForEach{$c\in C(T)$}
{
\lIf{\textsc{SpecialTPPTestSet}$(T,c,P)$}{$\mathcal U := \mathcal U \cup \{(S, ~ T\cup \{c\}, ~ U)\}$\;}
}
}
\If{$U\in I$}{
$P:=Q(S)Q(T)$\;
\ForEach{$c\in C(U)$}
{
\lIf{\textsc{SpecialTPPTestSet}$(U,c,P)$}{$\mathcal U := \mathcal U \cup \{(S, ~ T, ~ U\cup \{c\})\}$\;}
}
}
\Return{$\mathcal U$\;}
\caption{\textsc{UpgradeStepSet}$(S,T,U)$}
\end{algorithm}

The procedure \textsc{UpgradeStepSetBiggest} is similar to \textsc{UpgradeStepGroupBiggest}.

\section{Reducing TPP Triples as Input for the Upgrade Process}

\noindent In this section we deal with the problem, that $I(S,T,U)=\emptyset$ or $\mathcal U(S,T,U)=\emptyset$ for a given TPP triple $(S,T,U)$ of subgroups or subsets of $G$. We try to find an element $1\neq d \in S\cup T\cup U$ which we can \emph{delete} and \emph{restart} the upgrade process (for subsets). With $X_d\in \{S,T,U\}$ we denote the source of $d\in X_d$. We look at four strategies for building a TPP triple $T(d)$ from a given TPP triple $T$ by deleting~$d$:
\begin{enumerate}
 \item \textsc{RandomDelete}. Pick a random $1\neq d \in S\cup T\cup U$.
 \item \textsc{RandomWithRestrictionDelete}. Pick a random $1\neq d \in S\cup T\cup U$, such that $d\notin N$. Here $N$ denotes a given subset $N\subset S\cup T\cup U$ of elements that are \emph{not} allowed to be deleted, for example to make sure that no element of order 2 will be deleted or to systematically check all $d \in S\cup T\cup U$.
 \item \textsc{MaxTripleDelete}. Choose a $1\neq d \in S\cup T\cup U$ such that
\[|T(d')| \leq |T(d)| \qquad \forall d' \neq d.\]
If there are more than one such $d$'s, pick a random one.
\item \textsc{MaxQuotientDelete}. Choose a $1\neq d \in S\cup T\cup U$ such that
\[|Q(X_d)| - |Q(X_d\setminus\{d\})| \leq |Q(X_{d'})| - |Q(X_{d'}\setminus\{d'\})| \qquad \forall d' \neq d.\]
That means we try to maximize the reduction of $Q(S)$, $Q(T)$ or $Q(U)$. If there are more than one such $d$'s, pick a random one.
\end{enumerate}
Because we are only interested in matrix-matrix multiplication, we have to check that after the deletion of $d$, $|X_d|\geq 2$ still holds.

\section{Tests and Results}

\begin{table}
\fontsize{5}{5}\selectfont
\begin{tabular}{|l|ll|ll|l|}
\hline
Id($G$)  & $\beta_\mathrm{g}(G)$ & $\tens{n,p,m}$  & $\ell(G)$ & $\tens{\tilde n, \tilde p, \tilde m}$  & $\ell / \beta_\mathrm{g}$\\
\hline
$[ 52, 3 ]$ & 64 & $\langle 4, 4, 4 \rangle $ & 80 & $\langle 4, 4, \mathbf5 \rangle $ & 1.250 \\
$[ 78, 1 ]$ & 108 & $\langle 6, 6, 3 \rangle $ & 144 & $\langle 6, 6, \mathbf4 \rangle $ & 1.333 \\
$[ 104, 12 ]$ & 128 & $\langle 8, 4, 4 \rangle $ & 160 & $\langle 8, 4, \mathbf5 \rangle $ & 1.250 \\
$[ 156, 8 ]$ & 216 & $\langle 12, 6, 3 \rangle $ & 288 & $\langle 12, 6, \mathbf4 \rangle $ & 1.333 \\
$[ 156, 9 ]$ & 192 & $\langle 12, 4, 4 \rangle $ & 240 & $\langle 12, 4, \mathbf5 \rangle $ & 1.250 \\
$[ 156, 10 ]$ & 192 & $\langle 12, 4, 4 \rangle $ & 240 & $\langle 12, 4, \mathbf5 \rangle $ & 1.250 \\
$[ 186, 1 ]$ & 216 & $\langle 6, 6, 6 \rangle $ & 252 & $\langle 6, 6, \mathbf7 \rangle $ & 1.167 \\
$[ 192, 201 ]$ & 432 & $\langle 12, 6, 6 \rangle $ & 468 & $\langle \mathbf{13}, 6, 6 \rangle $ & 1.083 \\
$[ 196, 9 ]$ & 224 & $\langle 14, 4, 4 \rangle $ & 280 & $\langle 14, 4, \mathbf5 \rangle $ & 1.250 \\
$[ 208, 30 ]$ & 256 & $\langle 16, 4, 4 \rangle $ & 320 & $\langle 16, 4, \mathbf5 \rangle $ & 1.250 \\
$[ 208, 31 ]$ & 256 & $\langle 16, 4, 4 \rangle $ & 320 & $\langle 16, 4, \mathbf5 \rangle $ & 1.250 \\
$[ 208, 34 ]$ & 256 & $\langle 16, 4, 4 \rangle $ & 320 & $\langle 16, 4, \mathbf5 \rangle $ & 1.250 \\
$[ 208, 49 ]$ & 256 & $\langle 16, 4, 4 \rangle $ & 320 & $\langle 16, 4, \mathbf5 \rangle $ & 1.250 \\
$[ 234, 7 ]$ & 324 & $\langle 18, 6, 3 \rangle $ & 432 & $\langle 18, 6, \mathbf4 \rangle $ & 1.333 \\
$[ 234, 9 ]$ & 324 & $\langle 18, 6, 3 \rangle $ & 432 & $\langle 18, 6, \mathbf4 \rangle $ & 1.333 \\
$[ 260, 5 ]$ & 320 & $\langle 20, 4, 4 \rangle $ & 400 & $\langle 20, 4, \mathbf5 \rangle $ & 1.250 \\
$[ 260, 6 ]$ & 320 & $\langle 20, 4, 4 \rangle $ & 400 & $\langle 20, 4, \mathbf5 \rangle $ & 1.250 \\
$[ 301, 1 ]$ & 343 & $\langle 7, 7, 7 \rangle $ & 392 & $\langle 7, 7, \mathbf8 \rangle $ & 1.142 \\
$[ 310, 1 ]$ & 500 & $\langle 10, 10, 5 \rangle $ & 600 & $\langle 10, 10, \mathbf6 \rangle $ & 1.200 \\
$[ 312, 9 ]$ & 432 & $\langle 24, 6, 3 \rangle $ & 576 & $\langle 24, 6, \mathbf4 \rangle $ & 1.333 \\
$[ 312, 10 ]$ & 432 & $\langle 24, 6, 3 \rangle $ & 576 & $\langle 24, 6, \mathbf4 \rangle $ & 1.333 \\
$[ 312, 12 ]$ & 432 & $\langle 24, 6, 3 \rangle $ & 576 & $\langle 24, 6, \mathbf4 \rangle $ & 1.333 \\
$[ 312, 46 ]$ & 512 & $\langle 8, 8, 8 \rangle $ & 576 & $\langle 8, 8, \mathbf9 \rangle $ & 1.125 \\
$[ 312, 49 ]$ & 432 & $\langle 24, 6, 3 \rangle $ & 576 & $\langle 24, 6, \mathbf4 \rangle $ & 1.333 \\
$[ 312, 52 ]$ & 384 & $\langle 24, 4, 4 \rangle $ & 480 & $\langle 24, 4, \mathbf5 \rangle $ & 1.250 \\
$[ 312, 53 ]$ & 384 & $\langle 24, 4, 4 \rangle $ & 480 & $\langle 24, 4, \mathbf5 \rangle $ & 1.250 \\
$[ 328, 12 ]$ & 512 & $\langle 8, 8, 8 \rangle $ & 576 & $\langle \mathbf9, 8, 8 \rangle $ & 1.125 \\
$[ 351, 12 ]$ & 507 & $\langle 13, 13, 3 \rangle $ & 676 & $\langle 13, 13, \mathbf4 \rangle $ & 1.333 \\
$[ 364, 5 ]$ & 448 & $\langle 28, 4, 4 \rangle $ & 560 & $\langle 28, 4, \mathbf5 \rangle $ & 1.250 \\
$[ 364, 6 ]$ & 448 & $\langle 28, 4, 4 \rangle $ & 560 & $\langle 28, 4, \mathbf5 \rangle $ & 1.250 \\
$[ 372, 7 ]$ & 432 & $\langle 12, 6, 6 \rangle $ & 504 & $\langle 12, 6, \mathbf7 \rangle $ & 1.167 \\
$[ 384, 609 ]$ & 864 & $\langle 24, 6, 6 \rangle $ & 900 & $\langle \mathbf{25}, 6, 6 \rangle $ & 1.041 \\
$[ 384, 618 ]$ & 864 & $\langle 24, 6, 6 \rangle $ & 900 & $\langle \mathbf{25}, 6, 6 \rangle $ & 1.041 \\
$[ 390, 1 ]$ & 540 & $\langle 30, 6, 3 \rangle $ & 720 & $\langle 30, 6, \mathbf4 \rangle $ & 1.333 \\
$[ 390, 3 ]$ & 540 & $\langle 30, 6, 3 \rangle $ & 720 & $\langle 30, 6, \mathbf4 \rangle $ & 1.333 \\
$[ 392, 41 ]$ & 448 & $\langle 28, 4, 4 \rangle $ & 560 & $\langle 28, 4, \mathbf5 \rangle $ & 1.250 \\
$[ 400, 134 ]$ & 640 & $\langle 40, 4, 4 \rangle $ & 800 & $\langle 40, 4, \mathbf5 \rangle $ & 1.250 \\
$[ 405, 15 ]$ & 675 & $\langle 27, 5, 5 \rangle $ & 810 & $\langle 27, 5, \mathbf6 \rangle $ & 1.200 \\
$[ 410, 1 ]$ & 500 & $\langle 10, 10, 5 \rangle $ & 600 & $\langle 10, 10, \mathbf6 \rangle $ & 1.200 \\
$[ 416, 66 ]$ & 512 & $\langle 32, 4, 4 \rangle $ & 640 & $\langle 32, 4, \mathbf5 \rangle $ & 1.250 \\
$[ 416, 67 ]$ & 512 & $\langle 32, 4, 4 \rangle $ & 640 & $\langle 32, 4, \mathbf5 \rangle $ & 1.250 \\
$[ 416, 68 ]$ & 512 & $\langle 32, 4, 4 \rangle $ & 640 & $\langle 32, 4, \mathbf5 \rangle $ & 1.250 \\
$[ 416, 69 ]$ & 512 & $\langle 32, 4, 4 \rangle $ & 640 & $\langle 32, 4, \mathbf5 \rangle $ & 1.250 \\
$[ 416, 81 ]$ & 512 & $\langle 32, 4, 4 \rangle $ & 640 & $\langle 32, 4, \mathbf5 \rangle $ & 1.250 \\
$[ 416, 82 ]$ & 512 & $\langle 32, 4, 4 \rangle $ & 640 & $\langle 32, 4, \mathbf5 \rangle $ & 1.250 \\
$[ 416, 83 ]$ & 512 & $\langle 32, 4, 4 \rangle $ & 640 & $\langle 32, 4, \mathbf5 \rangle $ & 1.250 \\
$[ 416, 84 ]$ & 512 & $\langle 32, 4, 4 \rangle $ & 640 & $\langle 32, 4, \mathbf5 \rangle $ & 1.250 \\
$[ 416, 85 ]$ & 512 & $\langle 32, 4, 4 \rangle $ & 640 & $\langle 32, 4, \mathbf5 \rangle $ & 1.250 \\
$[ 416, 202 ]$ & 512 & $\langle 32, 4, 4 \rangle $ & 640 & $\langle 32, 4, \mathbf5 \rangle $ & 1.250 \\
$[ 416, 203 ]$ & 512 & $\langle 32, 4, 4 \rangle $ & 640 & $\langle 32, 4, \mathbf5 \rangle $ & 1.250 \\
$[ 416, 204 ]$ & 512 & $\langle 32, 4, 4 \rangle $ & 640 & $\langle 32, 4, \mathbf5 \rangle $ & 1.250 \\
$[ 416, 206 ]$ & 512 & $\langle 32, 4, 4 \rangle $ & 640 & $\langle 32, 4, \mathbf5 \rangle $ & 1.250 \\
$[ 416, 208 ]$ & 512 & $\langle 32, 4, 4 \rangle $ & 640 & $\langle 32, 4, \mathbf5 \rangle $ & 1.250 \\
$[ 416, 211 ]$ & 512 & $\langle 32, 4, 4 \rangle $ & 640 & $\langle 32, 4, \mathbf5 \rangle $ & 1.250 \\
$[ 416, 233 ]$ & 512 & $\langle 32, 4, 4 \rangle $ & 640 & $\langle 32, 4, \mathbf5 \rangle $ & 1.250 \\
$[ 444, 7 ]$ & 864 & $\langle 12, 12, 6 \rangle $ & 936 & $\langle \mathbf{13}, 12, 6 \rangle $ & 1.083 \\
$[ 468, 7 ]$ & 576 & $\langle 36, 4, 4 \rangle $ & 720 & $\langle 36, 4, \mathbf5 \rangle $ & 1.250 \\
$[ 468, 9 ]$ & 576 & $\langle 36, 4, 4 \rangle $ & 720 & $\langle 36, 4, \mathbf5 \rangle $ & 1.250 \\
$[ 468, 10 ]$ & 576 & $\langle 36, 4, 4 \rangle $ & 720 & $\langle 36, 4, \mathbf5 \rangle $ & 1.250 \\
$[ 468, 31 ]$ & 864 & $\langle 12, 12, 6 \rangle $ & 936 & $\langle \mathbf{13}, 12, 6 \rangle $ & 1.083 \\
$[ 468, 33 ]$ & 648 & $\langle 36, 6, 3 \rangle $ & 864 & $\langle 36, 6, \mathbf4 \rangle $ & 1.333 \\
$[ 468, 35 ]$ & 648 & $\langle 36, 6, 3 \rangle $ & 864 & $\langle 36, 6, \mathbf4 \rangle $ & 1.333 \\
$[ 468, 36 ]$ & 576 & $\langle 36, 4, 4 \rangle $ & 720 & $\langle 36, 4, \mathbf5 \rangle $ & 1.250 \\
$[ 468, 37 ]$ & 576 & $\langle 36, 4, 4 \rangle $ & 720 & $\langle 36, 4, \mathbf5 \rangle $ & 1.250 \\
$[ 468, 38 ]$ & 576 & $\langle 36, 4, 4 \rangle $ & 720 & $\langle 36, 4, \mathbf5 \rangle $ & 1.250 \\
$[ 520, 38 ]$ & 640 & $\langle 40, 4, 4 \rangle $ & 800 & $\langle 40, 4, \mathbf5 \rangle $ & 1.250 \\
$[ 520, 39 ]$ & 640 & $\langle 40, 4, 4 \rangle $ & 800 & $\langle 40, 4, \mathbf5 \rangle $ & 1.250 \\
$[ 520, 44 ]$ & 640 & $\langle 40, 4, 4 \rangle $ & 800 & $\langle 40, 4, \mathbf5 \rangle $ & 1.250 \\
$[ 546, 1 ]$ & 756 & $\langle 42, 6, 3 \rangle $ & 1008 & $\langle 42, 6, \mathbf4 \rangle $ & 1.333 \\
$[ 546, 3 ]$ & 756 & $\langle 42, 6, 3 \rangle $ & 1008 & $\langle 42, 6, \mathbf4 \rangle $ & 1.333 \\
$[ 546, 5 ]$ & 756 & $\langle 42, 6, 3 \rangle $ & 1008 & $\langle 42, 6, \mathbf4 \rangle $ & 1.333 \\
$[ 546, 6 ]$ & 756 & $\langle 42, 6, 3 \rangle $ & 1008 & $\langle 42, 6, \mathbf4 \rangle $ & 1.333 \\
$[ 558, 7 ]$ & 648 & $\langle 18, 6, 6 \rangle $ & 756 & $\langle 18, 6, \mathbf7 \rangle $ & 1.167 \\
$[ 572, 5 ]$ & 704 & $\langle 44, 4, 4 \rangle $ & 880 & $\langle 44, 4, \mathbf5 \rangle $ & 1.250 \\
$[ 572, 6 ]$ & 704 & $\langle 44, 4, 4 \rangle $ & 880 & $\langle 44, 4, \mathbf5 \rangle $ & 1.250 \\
$[ 588, 47 ]$ & 672 & $\langle 42, 4, 4 \rangle $ & 840 & $\langle 42, \mathbf5, 4 \rangle $ & 1.250 \\
$[ 588, 48 ]$ & 672 & $\langle 42, 4, 4 \rangle $ & 840 & $\langle 42, 4, \mathbf5 \rangle $ & 1.250 \\
$[ 588, 49 ]$ & 672 & $\langle 42, 4, 4 \rangle $ & 840 & $\langle 42, 4, \mathbf5 \rangle $ & 1.250 \\
$[ 600, 162 ]$ & 1280 & $\langle 20, 8, 8 \rangle $ & 1344 & $\langle \mathbf{21}, 8, 8 \rangle $ & 1.050 \\
$[ 602, 1 ]$ & 1372 & $\langle 14, 14, 7 \rangle $ & 1568 & $\langle 14, 14, \mathbf8 \rangle $ & 1.142 \\
$[ 602, 2 ]$ & 686 & $\langle 14, 7, 7 \rangle $ & 784 & $\langle 14, 7, \mathbf8 \rangle $ & 1.142 \\
$[ 610, 1 ]$ & 1000 & $\langle 10, 10, 10 \rangle $ & 1100 & $\langle \mathbf{11}, 10, 10 \rangle $ & 1.100 \\
$[ 620, 7 ]$ & 1000 & $\langle 20, 10, 5 \rangle $ & 1200 & $\langle 20, 10, \mathbf6 \rangle $ & 1.200 \\
$[ 624, 8 ]$ & 864 & $\langle 48, 6, 3 \rangle $ & 1152 & $\langle 48, 6, \mathbf4 \rangle $ & 1.333 \\
$[ 624, 9 ]$ & 864 & $\langle 48, 6, 3 \rangle $ & 1152 & $\langle 48, 6, \mathbf4 \rangle $ & 1.333 \\
$[ 624, 10 ]$ & 864 & $\langle 48, 6, 3 \rangle $ & 1152 & $\langle 48, 6, \mathbf4 \rangle $ & 1.333 \\
$[ 624, 11 ]$ & 864 & $\langle 48, 6, 3 \rangle $ & 1152 & $\langle 48, 6, \mathbf4 \rangle $ & 1.333 \\
$[ 624, 18 ]$ & 864 & $\langle 48, 6, 3 \rangle $ & 1152 & $\langle 48, 6, \mathbf4 \rangle $ & 1.333 \\
$[ 624, 19 ]$ & 864 & $\langle 48, 6, 3 \rangle $ & 1152 & $\langle 48, 6, \mathbf4 \rangle $ & 1.333 \\
$[ 624, 21 ]$ & 864 & $\langle 48, 6, 3 \rangle $ & 1152 & $\langle 48, 6, \mathbf4 \rangle $ & 1.333 \\
$[ 624, 23 ]$ & 864 & $\langle 12, 12, 6 \rangle $ & 1008 & $\langle 12, 12, \mathbf7 \rangle $ & 1.167 \\
$[ 624, 117 ]$ & 864 & $\langle 48, 6, 3 \rangle $ & 1152 & $\langle 48, 6, \mathbf4 \rangle $ & 1.333 \\
$[ 624, 118 ]$ & 864 & $\langle 48, 6, 3 \rangle $ & 1152 & $\langle 48, 6, \mathbf4 \rangle $ & 1.333 \\
$[ 624, 124 ]$ & 768 & $\langle 48, 4, 4 \rangle $ & 960 & $\langle 48, 4, \mathbf5 \rangle $ & 1.250 \\
$[ 624, 126 ]$ & 768 & $\langle 48, 4, 4 \rangle $ & 960 & $\langle 48, 4, \mathbf5 \rangle $ & 1.250 \\
$[ 624, 138 ]$ & 864 & $\langle 48, 6, 3 \rangle $ & 1152 & $\langle 48, 6, \mathbf4 \rangle $ & 1.333 \\
$[ 624, 139 ]$ & 864 & $\langle 48, 6, 3 \rangle $ & 1152 & $\langle 48, 6, \mathbf4 \rangle $ & 1.333 \\
$[ 624, 140 ]$ & 864 & $\langle 48, 6, 3 \rangle $ & 1152 & $\langle 48, 6, \mathbf4 \rangle $ & 1.333 \\
$[ 624, 141 ]$ & 864 & $\langle 48, 6, 3 \rangle $ & 1152 & $\langle 48, 6, \mathbf4 \rangle $ & 1.333 \\
$[ 624, 142 ]$ & 864 & $\langle 48, 6, 3 \rangle $ & 1152 & $\langle 48, 6, \mathbf4 \rangle $ & 1.333 \\
$[ 624, 143 ]$ & 864 & $\langle 48, 6, 3 \rangle $ & 1152 & $\langle 48, 6, \mathbf4 \rangle $ & 1.333 \\
$[ 624, 144 ]$ & 864 & $\langle 48, 6, 3 \rangle $ & 1152 & $\langle 48, 6, \mathbf4 \rangle $ & 1.333 \\
$[ 624, 146 ]$ & 864 & $\langle 48, 6, 3 \rangle $ & 1152 & $\langle 48, 6, \mathbf4 \rangle $ & 1.333 \\
$[ 624, 155 ]$ & 768 & $\langle 48, 4, 4 \rangle $ & 960 & $\langle 48, 4, \mathbf5 \rangle $ & 1.250 \\
$[ 624, 156 ]$ & 768 & $\langle 48, 4, 4 \rangle $ & 960 & $\langle 48, 4, \mathbf5 \rangle $ & 1.250 \\
$[ 624, 159 ]$ & 768 & $\langle 48, 4, 4 \rangle $ & 960 & $\langle 48, 4, \mathbf5 \rangle $ & 1.250 \\
$[ 624, 162 ]$ & 768 & $\langle 48, 4, 4 \rangle $ & 960 & $\langle 48, 4, \mathbf5 \rangle $ & 1.250 \\
$[ 624, 163 ]$ & 768 & $\langle 48, 4, 4 \rangle $ & 960 & $\langle 48, 4, \mathbf5 \rangle $ & 1.250 \\
$[ 624, 166 ]$ & 768 & $\langle 48, 4, 4 \rangle $ & 960 & $\langle 48, 4, \mathbf5 \rangle $ & 1.250 \\
$[ 624, 238 ]$ & 1152 & $\langle 12, 12, 8 \rangle $ & 1296 & $\langle 12, 12, \mathbf9 \rangle $ & 1.125 \\\hline
\end{tabular}
\begin{tabular}{|l|ll|ll|l|}
\hline
Id($G$)  & $\beta_\mathrm{g}(G)$ & $\tens{n,p,m}$  & $\ell(G)$ & $\tens{\tilde n, \tilde p, \tilde m}$  & $\ell / \beta_\mathrm{g}$\\
\hline
$[ 624, 240 ]$ & 1152 & $\langle 12, 12, 8 \rangle $ & 1296 & $\langle 12, 12, \mathbf9 \rangle $ & 1.125 \\
$[ 624, 243 ]$ & 1024 & $\langle 16, 8, 8 \rangle $ & 1152 & $\langle 16, 8, \mathbf9 \rangle $ & 1.125 \\
$[ 624, 246 ]$ & 864 & $\langle 48, 6, 3 \rangle $ & 1152 & $\langle 48, 6, \mathbf4 \rangle $ & 1.333 \\
$[ 624, 249 ]$ & 768 & $\langle 48, 4, 4 \rangle $ & 960 & $\langle 48, 4, \mathbf5 \rangle $ & 1.250 \\
$[ 624, 250 ]$ & 768 & $\langle 48, 4, 4 \rangle $ & 960 & $\langle 48, 4, \mathbf5 \rangle $ & 1.250 \\
$[ 656, 50 ]$ & 1024 & $\langle 16, 8, 8 \rangle $ & 1088 & $\langle \mathbf{17}, 8, 8 \rangle $ & 1.062 \\
$[ 676, 3 ]$ & 832 & $\langle 52, 4, 4 \rangle $ & 1040 & $\langle 52, 4, \mathbf5 \rangle $ & 1.250 \\
$[ 676, 9 ]$ & 832 & $\langle 52, 4, 4 \rangle $ & 1040 & $\langle 52, 4, \mathbf5 \rangle $ & 1.250 \\
$[ 676, 10 ]$ & 832 & $\langle 52, 4, 4 \rangle $ & 1040 & $\langle 52, 4, \mathbf5 \rangle $ & 1.250 \\
$[ 676, 11 ]$ & 832 & $\langle 52, 4, 4 \rangle $ & 1040 & $\langle 52, 4, \mathbf5 \rangle $ & 1.250 \\
$[ 676, 12 ]$ & 832 & $\langle 52, 4, 4 \rangle $ & 1040 & $\langle 52, 4, \mathbf5 \rangle $ & 1.250 \\
$[ 684, 16 ]$ & 972 & $\langle 36, 9, 3 \rangle $ & 1296 & $\langle 36, 9, \mathbf4 \rangle $ & 1.333 \\
$[ 686, 5 ]$ & 1372 & $\langle 14, 14, 7 \rangle $ & 1568 & $\langle 14, 14, \mathbf8 \rangle $ & 1.142 \\
$[ 702, 7 ]$ & 972 & $\langle 54, 6, 3 \rangle $ & 1296 & $\langle 54, 6, \mathbf4 \rangle $ & 1.333 \\
$[ 702, 8 ]$ & 972 & $\langle 54, 6, 3 \rangle $ & 1296 & $\langle 54, 6, \mathbf4 \rangle $ & 1.333 \\
$[ 702, 9 ]$ & 972 & $\langle 54, 6, 3 \rangle $ & 1296 & $\langle 54, 6, \mathbf4 \rangle $ & 1.333 \\
$[ 702, 12 ]$ & 972 & $\langle 54, 6, 3 \rangle $ & 1296 & $\langle 54, 6, \mathbf4 \rangle $ & 1.333 \\
$[ 702, 18 ]$ & 972 & $\langle 54, 6, 3 \rangle $ & 1296 & $\langle 54, 6, \mathbf4 \rangle $ & 1.333 \\
$[ 702, 19 ]$ & 972 & $\langle 54, 6, 3 \rangle $ & 1296 & $\langle 54, 6, \mathbf4 \rangle $ & 1.333 \\
$[ 702, 20 ]$ & 972 & $\langle 54, 6, 3 \rangle $ & 1296 & $\langle 54, 6, \mathbf4 \rangle $ & 1.333 \\
$[ 702, 48 ]$ & 1014 & $\langle 26, 13, 3 \rangle $ & 1352 & $\langle 26, 13, \mathbf4 \rangle $ & 1.333 \\
$[ 702, 49 ]$ & 972 & $\langle 54, 6, 3 \rangle $ & 1296 & $\langle 54, 6, \mathbf4 \rangle $ & 1.333 \\
$[ 702, 51 ]$ & 972 & $\langle 54, 6, 3 \rangle $ & 1296 & $\langle 54, 6, \mathbf4 \rangle $ & 1.333 \\
$[ 702, 53 ]$ & 972 & $\langle 54, 6, 3 \rangle $ & 1296 & $\langle 54, 6, \mathbf4 \rangle $ & 1.333 \\
$[ 710, 1 ]$ & 1000 & $\langle 10, 10, 10 \rangle $ & 1100 & $\langle 10, 10, \mathbf{11} \rangle $ & 1.100 \\
$[ 726, 6 ]$ & 792 & $\langle 22, 6, 6 \rangle $ & 924 & $\langle 22, 6, \mathbf7 \rangle $ & 1.167 \\
$[ 728, 32 ]$ & 896 & $\langle 56, 4, 4 \rangle $ & 1120 & $\langle 56, 4, \mathbf5 \rangle $ & 1.250 \\
$[ 728, 34 ]$ & 896 & $\langle 56, 4, 4 \rangle $ & 1120 & $\langle 56, 4, \mathbf5 \rangle $ & 1.250 \\
$[ 728, 35 ]$ & 896 & $\langle 56, 4, 4 \rangle $ & 1120 & $\langle 56, 4, \mathbf5 \rangle $ & 1.250 \\
$[ 732, 7 ]$ & 864 & $\langle 12, 12, 6 \rangle $ & 1008 & $\langle 12, 12, \mathbf7 \rangle $ & 1.167 \\
$[ 744, 8 ]$ & 864 & $\langle 24, 6, 6 \rangle $ & 1008 & $\langle 24, 6, \mathbf7 \rangle $ & 1.167 \\
$[ 744, 9 ]$ & 864 & $\langle 24, 6, 6 \rangle $ & 1008 & $\langle 24, 6, \mathbf7 \rangle $ & 1.167 \\
$[ 744, 11 ]$ & 864 & $\langle 24, 6, 6 \rangle $ & 1008 & $\langle 24, 6, \mathbf7 \rangle $ & 1.167 \\
$[ 744, 44 ]$ & 864 & $\langle 24, 6, 6 \rangle $ & 1008 & $\langle 24, 6, \mathbf7 \rangle $ & 1.167 \\
$[ 750, 28 ]$ & 900 & $\langle 25, 6, 6 \rangle $ & 1050 & $\langle 25, 6, \mathbf7 \rangle $ & 1.167 \\
$[ 780, 20 ]$ & 1080 & $\langle 60, 6, 3 \rangle $ & 1440 & $\langle 60, 6, \mathbf4 \rangle $ & 1.333 \\
$[ 780, 21 ]$ & 1080 & $\langle 60, 6, 3 \rangle $ & 1440 & $\langle 60, 6, \mathbf4 \rangle $ & 1.333 \\
$[ 780, 23 ]$ & 1080 & $\langle 60, 6, 3 \rangle $ & 1440 & $\langle 60, 6, \mathbf4 \rangle $ & 1.333 \\
$[ 780, 24 ]$ & 960 & $\langle 60, 4, 4 \rangle $ & 1200 & $\langle 60, 4, \mathbf5 \rangle $ & 1.250 \\
$[ 780, 29 ]$ & 960 & $\langle 60, 4, 4 \rangle $ & 1200 & $\langle 60, 4, \mathbf5 \rangle $ & 1.250 \\
$[ 780, 30 ]$ & 960 & $\langle 60, 4, 4 \rangle $ & 1200 & $\langle 60, 4, \mathbf5 \rangle $ & 1.250 \\
$[ 780, 35 ]$ & 960 & $\langle 60, 4, 4 \rangle $ & 1200 & $\langle 60, 4, \mathbf5 \rangle $ & 1.250 \\
$[ 784, 105 ]$ & 896 & $\langle 56, 4, 4 \rangle $ & 1120 & $\langle 56, 4, \mathbf5 \rangle $ & 1.250 \\
$[ 784, 106 ]$ & 896 & $\langle 56, 4, 4 \rangle $ & 1120 & $\langle 56, 4, \mathbf5 \rangle $ & 1.250 \\
$[ 784, 110 ]$ & 896 & $\langle 16, 14, 4 \rangle $ & 1120 & $\langle 16, 14, \mathbf5 \rangle $ & 1.250 \\
$[ 784, 113 ]$ & 896 & $\langle 56, 4, 4 \rangle $ & 1120 & $\langle 56, 4, \mathbf5 \rangle $ & 1.250 \\
$[ 784, 114 ]$ & 896 & $\langle 56, 4, 4 \rangle $ & 1120 & $\langle 56, 4, \mathbf5 \rangle $ & 1.250 \\
$[ 784, 117 ]$ & 896 & $\langle 56, 4, 4 \rangle $ & 1120 & $\langle 56, 4, \mathbf5 \rangle $ & 1.250 \\
$[ 784, 124 ]$ & 896 & $\langle 56, 4, 4 \rangle $ & 1120 & $\langle 56, 4, \mathbf5 \rangle $ & 1.250 \\
$[ 784, 125 ]$ & 896 & $\langle 56, 4, 4 \rangle $ & 1120 & $\langle 56, 4, \mathbf5 \rangle $ & 1.250 \\
$[ 784, 126 ]$ & 896 & $\langle 56, 4, 4 \rangle $ & 1120 & $\langle 56, 4, \mathbf5 \rangle $ & 1.250 \\
$[ 820, 8 ]$ & 1000 & $\langle 20, 10, 5 \rangle $ & 1200 & $\langle 20, 10, \mathbf6 \rangle $ & 1.200 \\
$[ 832, 181 ]$ & 1024 & $\langle 64, 4, 4 \rangle $ & 1280 & $\langle 64, 4, \mathbf5 \rangle $ & 1.250 \\
$[ 832, 182 ]$ & 1024 & $\langle 64, 4, 4 \rangle $ & 1280 & $\langle 64, 4, \mathbf5 \rangle $ & 1.250 \\
$[ 832, 183 ]$ & 1024 & $\langle 64, 4, 4 \rangle $ & 1280 & $\langle 64, 4, \mathbf5 \rangle $ & 1.250 \\
$[ 832, 184 ]$ & 1024 & $\langle 64, 4, 4 \rangle $ & 1280 & $\langle 64, 4, \mathbf5 \rangle $ & 1.250 \\
$[ 832, 185 ]$ & 1024 & $\langle 64, 4, 4 \rangle $ & 1280 & $\langle 64, 4, \mathbf5 \rangle $ & 1.250 \\
$[ 832, 186 ]$ & 1024 & $\langle 64, 4, 4 \rangle $ & 1280 & $\langle 64, 4, \mathbf5 \rangle $ & 1.250 \\
$[ 832, 187 ]$ & 1024 & $\langle 64, 4, 4 \rangle $ & 1280 & $\langle 64, 4, \mathbf5 \rangle $ & 1.250 \\
$[ 832, 188 ]$ & 1024 & $\langle 64, 4, 4 \rangle $ & 1280 & $\langle 64, 4, \mathbf5 \rangle $ & 1.250 \\
$[ 832, 200 ]$ & 1024 & $\langle 64, 4, 4 \rangle $ & 1280 & $\langle 64, 4, \mathbf5 \rangle $ & 1.250 \\
$[ 832, 201 ]$ & 1024 & $\langle 64, 4, 4 \rangle $ & 1280 & $\langle 64, 4, \mathbf5 \rangle $ & 1.250 \\
$[ 832, 203 ]$ & 1024 & $\langle 64, 4, 4 \rangle $ & 1280 & $\langle 64, 4, \mathbf5 \rangle $ & 1.250 \\
$[ 832, 206 ]$ & 1024 & $\langle 64, 4, 4 \rangle $ & 1280 & $\langle 64, 4, \mathbf5 \rangle $ & 1.250 \\
$[ 832, 207 ]$ & 1024 & $\langle 64, 4, 4 \rangle $ & 1280 & $\langle 64, 4, \mathbf5 \rangle $ & 1.250 \\
$[ 832, 212 ]$ & 1024 & $\langle 64, 4, 4 \rangle $ & 1280 & $\langle 64, 4, \mathbf5 \rangle $ & 1.250 \\
$[ 832, 213 ]$ & 1024 & $\langle 64, 4, 4 \rangle $ & 1280 & $\langle 64, 4, \mathbf5 \rangle $ & 1.250 \\
$[ 832, 230 ]$ & 1024 & $\langle 64, 4, 4 \rangle $ & 1280 & $\langle 64, 4, \mathbf5 \rangle $ & 1.250 \\
$[ 832, 231 ]$ & 1024 & $\langle 64, 4, 4 \rangle $ & 1280 & $\langle 64, 4, \mathbf5 \rangle $ & 1.250 \\
$[ 832, 232 ]$ & 1024 & $\langle 64, 4, 4 \rangle $ & 1280 & $\langle 64, 4, \mathbf5 \rangle $ & 1.250 \\
$[ 832, 233 ]$ & 1024 & $\langle 64, 4, 4 \rangle $ & 1280 & $\langle 64, 4, \mathbf5 \rangle $ & 1.250 \\
$[ 832, 237 ]$ & 1024 & $\langle 64, 4, 4 \rangle $ & 1280 & $\langle 64, 4, \mathbf5 \rangle $ & 1.250 \\
$[ 832, 238 ]$ & 1024 & $\langle 64, 4, 4 \rangle $ & 1280 & $\langle 64, 4, \mathbf5 \rangle $ & 1.250 \\
$[ 832, 240 ]$ & 1024 & $\langle 64, 4, 4 \rangle $ & 1280 & $\langle 64, 4, \mathbf5 \rangle $ & 1.250 \\
$[ 832, 241 ]$ & 1024 & $\langle 64, 4, 4 \rangle $ & 1280 & $\langle 64, 4, \mathbf5 \rangle $ & 1.250 \\
$[ 832, 242 ]$ & 1024 & $\langle 64, 4, 4 \rangle $ & 1280 & $\langle 64, 4, \mathbf5 \rangle $ & 1.250 \\
$[ 832, 245 ]$ & 1024 & $\langle 64, 4, 4 \rangle $ & 1280 & $\langle 64, 4, \mathbf5 \rangle $ & 1.250 \\
$[ 832, 246 ]$ & 1024 & $\langle 64, 4, 4 \rangle $ & 1280 & $\langle 64, 4, \mathbf5 \rangle $ & 1.250 \\
$[ 832, 254 ]$ & 1024 & $\langle 64, 4, 4 \rangle $ & 1280 & $\langle 64, 4, \mathbf5 \rangle $ & 1.250 \\
$[ 832, 258 ]$ & 1024 & $\langle 64, 4, 4 \rangle $ & 1280 & $\langle 64, 4, \mathbf5 \rangle $ & 1.250 \\
$[ 832, 264 ]$ & 1024 & $\langle 64, 4, 4 \rangle $ & 1280 & $\langle 64, 4, \mathbf5 \rangle $ & 1.250 \\
$[ 858, 1 ]$ & 1188 & $\langle 66, 6, 3 \rangle $ & 1584 & $\langle 66, 6, \mathbf4 \rangle $ & 1.333 \\
$[ 858, 3 ]$ & 1188 & $\langle 66, 6, 3 \rangle $ & 1584 & $\langle 66, 6, \mathbf4 \rangle $ & 1.333 \\
$[ 876, 7 ]$ & 1728 & $\langle 12, 12, 12 \rangle $ & 1872 & $\langle 12, \mathbf{13}, 12 \rangle $ & 1.083 \\
$[ 884, 6 ]$ & 1088 & $\langle 68, 4, 4 \rangle $ & 1360 & $\langle 68, 4, \mathbf5 \rangle $ & 1.250 \\
$[ 884, 7 ]$ & 1088 & $\langle 68, 4, 4 \rangle $ & 1360 & $\langle 68, 4, \mathbf5 \rangle $ & 1.250 \\
$[ 884, 8 ]$ & 1088 & $\langle 68, 4, 4 \rangle $ & 1360 & $\langle 68, 4, \mathbf5 \rangle $ & 1.250 \\
$[ 884, 9 ]$ & 1088 & $\langle 68, 4, 4 \rangle $ & 1360 & $\langle 68, 4, \mathbf5 \rangle $ & 1.250 \\
$[ 888, 45 ]$ & 1728 & $\langle 24, 12, 6 \rangle $ & 1800 & $\langle \mathbf{25}, 12, 6 \rangle $ & 1.041 \\
$[ 903, 1 ]$ & 1323 & $\langle 21, 21, 3 \rangle $ & 1764 & $\langle 21, 21, \mathbf4 \rangle $ & 1.333 \\
$[ 903, 2 ]$ & 1029 & $\langle 21, 7, 7 \rangle $ & 1176 & $\langle 21, 7, \mathbf8 \rangle $ & 1.142 \\
$[ 915, 1 ]$ & 1125 & $\langle 15, 15, 5 \rangle $ & 1350 & $\langle 15, 15, \mathbf6 \rangle $ & 1.200 \\
$[ 930, 3 ]$ & 1500 & $\langle 30, 10, 5 \rangle $ & 1800 & $\langle 30, 10, \mathbf6 \rangle $ & 1.200 \\
$[ 930, 5 ]$ & 1500 & $\langle 30, 10, 5 \rangle $ & 1800 & $\langle 30, 10, \mathbf6 \rangle $ & 1.200 \\
$[ 930, 6 ]$ & 1080 & $\langle 30, 6, 6 \rangle $ & 1260 & $\langle 30, 6, \mathbf7 \rangle $ & 1.167 \\
$[ 930, 8 ]$ & 1080 & $\langle 30, 6, 6 \rangle $ & 1260 & $\langle 30, 6, \mathbf7 \rangle $ & 1.167 \\
$[ 968, 35 ]$ & 1408 & $\langle 22, 8, 8 \rangle $ & 1584 & $\langle 22, \mathbf9, 8 \rangle $ & 1.125 \\
$[ 968, 36 ]$ & 1408 & $\langle 22, 8, 8 \rangle $ & 1584 & $\langle 22, 8, \mathbf9 \rangle $ & 1.125 \\
$[ 968, 37 ]$ & 1408 & $\langle 22, 8, 8 \rangle $ & 1584 & $\langle 22, 8, \mathbf9 \rangle $ & 1.125 \\
$[ 979, 1 ]$ & 1331 & $\langle 11, 11, 11 \rangle $ & 1452 & $\langle 11, 11, \mathbf{12} \rangle $ & 1.090 \\
$[ 980, 18 ]$ & 1120 & $\langle 70, 4, 4 \rangle $ & 1400 & $\langle 70, \mathbf5, 4 \rangle $ & 1.250 \\
$[ 980, 23 ]$ & 1120 & $\langle 70, 4, 4 \rangle $ & 1400 & $\langle 70, 4, \mathbf5 \rangle $ & 1.250 \\
$[ 980, 24 ]$ & 1120 & $\langle 70, 4, 4 \rangle $ & 1400 & $\langle 70, 4, \mathbf5 \rangle $ & 1.250 \\
$[ 980, 27 ]$ & 1120 & $\langle 70, 4, 4 \rangle $ & 1400 & $\langle 70, 4, \mathbf5 \rangle $ & 1.250 \\
$[ 980, 28 ]$ & 1120 & $\langle 70, 4, 4 \rangle $ & 1400 & $\langle 70, 4, \mathbf5 \rangle $ & 1.250 \\
$[ 984, 33 ]$ & 1536 & $\langle 24, 8, 8 \rangle $ & 1600 & $\langle \mathbf{25}, 8, 8 \rangle $ & 1.041 \\
$[ 984, 34 ]$ & 1536 & $\langle 24, 8, 8 \rangle $ & 1600 & $\langle \mathbf{25}, 8, 8 \rangle $ & 1.041 \\
$[ 988, 5 ]$ & 1216 & $\langle 76, 4, 4 \rangle $ & 1520 & $\langle 76, 4, \mathbf5 \rangle $ & 1.250 \\
$[ 988, 6 ]$ & 1216 & $\langle 76, 4, 4 \rangle $ & 1520 & $\langle 76, 4, \mathbf5 \rangle $ & 1.250 \\
$[ 994, 1 ]$ & 1372 & $\langle 14, 14, 7 \rangle $ & 1568 & $\langle 14, 14, \mathbf8 \rangle $ & 1.142 \\\hline
\end{tabular}
\normalsize
\caption{Results of the upgrade process for subgroups with input from \cite{HedtkeMurthy2011}.}
\label{tab:BiggGrp}
\end{table}

\noindent We start with the results from the upgrade process for subgroups. We use the results of the brute-force search from \cite{HedtkeMurthy2011}, located at \texttt{http://www2.informatik.uni-halle.de/da/hedtke/tpp/}. This means we try to upgrade more than 48000 maximal subgroup TPP triples for groups of order up to 1000. The results of \textsc{UpgradeStepGroupBiggest} are can be found in Table~\ref{tab:BiggGrp}. The table shows the IdSmallGroup of the SmallGroups Library of \textsf{GAP}\nocite{GAP4.4.12}, the TPP subgroup capacity $\beta_\mathrm{g}$, the realized problem $\langle n, p, m\rangle$ by subgroups, the new lower bound $\ell$ for $\beta$ found by the upgrade process, the new realized problem $\langle \tilde n, \tilde p, \tilde m\rangle$ and the quality of the new result measured by $\ell / \beta_\mathrm{g}$. There are circa $2000$ input TPP triples for the upgrade process (all nontrivial TPP subgroup triples in the tables of \cite{HedtkeMurthy2011}). The upgrade process for subgroups results in new TPP triples for exactly $220$ of these input triples. The upgraded part of the triple is printed in boldface in the table. In the best cases we achieve $\ell / \beta_\mathrm{g} = 4/3$, which means that we were able to increase the lower bound for $\beta$ by $33$ \%. None of the upgraded triples is of a size bigger than $D_3$.

Our second test is a combination of one upgrade step for subgroup triples and as many upgrade steps for subset triples as possible. In the first step we compute \emph{all} upgrades of subgroup triples for the groups listed in Table~\ref{tab:BiggGrp}. The output is a list of subset TPP triples. Now we compute \emph{all} upgrades of these subset triples. And after this we compute all upgrades of the upgraded triples, and so on.
The results of this process are shown in Table~\ref{tab:BiggSet}. The new lower bound $\ell'$ is the result of the iterative upgrade process.
If the iterative process achieves a better bound, the new value is printed in boldface.
In the best cases we achieve $\ell' / \beta_\mathrm{g} = 2$, which means that we were able to increase the lower bound for $\beta$ by $100$ \%. None of the upgraded triples is of a size bigger than $D_3$.
The memory requirements of this process are very high. One star (*) in the table indicates, that the iterative process abborted while using 7 GB of RAM. Two stars (**) indicates, that the process abborted with 31 GB of RAM. 

Finally we look at the case that no upgrade is possible (all nontrivial subgroup TPP triples from \cite{HedtkeMurthy2011} except the groups from Table~\ref{tab:BiggGrp}). We reduce the given maximal triple with one of the methods described in the previous section and start two upgrade steps for subset triples (because only one upgrade step after one reduction step makes no sense). Because the reduction methods contain a random part we repeat the test three times. The results are shown in Table~\ref{tab:Red}. Here $\ell_\mathrm{X}$ denotes the lower bound achieved after one reduction step with method X and two upgrade steps. In the best cases we achieve $\ell_\mathrm{X} / \beta_\mathrm{g} = 4/3$, which means that we were able to increase the lower bound for $\beta$ by $33$ \%. None of the upgraded triples is of a size bigger than $D_3$. If this process leads to a new lower bound the value is printed in boldface. We reduction methods are bases on different heuristics. Therefore we
also counted the successful upgrades after one step with the different reduction methods. With \textsc{MaxTripleDelete} we achieve the most sucessful upgrades (43 times), followed by the \textsc{RandomDelete} method (23 times) and \textsc{MaxQuotientDelete} (only 9 times). Thus we recommend the reduction method \textsc{MaxTripleDelete}.

~

The tests above tell us, that the presented upgrade and reduction methods are well suited to enlarge the output of the bruteforce-search for subgroups triples. Some researchers believe that triples of subgroups will never lead to a new nontrivial upper bound for $\omega$ in the context of the (TPP). But now we can use the fast end efficient search methods for subgroup triples and upgrade the output with our new methods.

The next steps in the context of upgrading und reducing TPP triples could be efficient implementations of the iterative upgrading process (as mentioned above, the memory requirements are enormous) and a good strategy of how to combine the (one or more) reduction steps and iterative upgrading process.

\begin{table}
\fontsize{5}{6}\selectfont
\begin{tabular}{|lr|rrr|l|}
\hline
Id($G$)  & $D_3(G)$ &  $\beta_\mathrm{g}(G)$ & $\ell(G)$ & $\ell'(G)$ & $\ell'/\beta_\mathrm{g}$\\
\hline
$[ 104, 12 ]$ & 	392 & 	128 & 	160 & 	160 		& 1.250\\
$[ 156, 9 ]$ & 		588 & 	192 & 	240 & 	240 		& 1.250\\
$[ 156, 10 ]$ & 	596 & 	192 & 	240 & 	240 		& 1.250\\
$[ 186, 1 ]$ & 		1086 &  216 & 	252 & 	\textbf{432} 	& 2.000\\
$[ 192, 201 ]$ & 	876 & 	432 & 	468 & 	\textbf{504} 	& 1.167\\
$[ 196, 9 ]$ & 		676 & 	224 & 	280 & 	\textbf{336}	& 1.500\\
$[ 208, 30 ]$ & 	784 & 	256 & 	320 & 	320 		& 1.250\\
$[ 208, 31 ]$ & 	792 &	256 & 	320 & 	320 		& 1.250\\
$[ 208, 34 ]$ & 	792 & 	256 & 	320 & 	320 		& 1.250\\
$[ 208, 49 ]$ & 	784	& 	256 & 	320 & 	320 	& 1.250\\
$[ 234, 7 ]$ & 		1314 &	324 & 	432 & 	432 		& 1.333\\
$[ 234, 9 ]$ & 		1326 &	324 & 	432 & 	432 		& 1.333\\
$[ 260, 5 ]$ & 		980 & 	320 & 	400 & 	400		& 1.250\\
$[ 260, 6 ]$ & 		996 &	320 & 	400 &  	400		& 1.250\\
$[ 301, 1 ]$ & 		2065 &	343 & 	392 &  	\textbf{539}	& 1.571\\
$[ 310, 1 ]$ & 		3010 &	500 &	600 &  	\textbf{800}	& 1.600\\
$[ 312, 9 ]$ & 		1752 &	432 & 	576 &  	576		& 1.333\\
$[ 312, 10 ]$ & 	1764 &	432 &  	576 &  	576		& 1.333\\
$[ 312, 12 ]$ & 	1764 &	432 & 	576 &  	576		& 1.333\\
$[ 312, 46 ]$ & 	1960 &	512 & 	576 & 	\textbf{640}	& 1.250\\
$[ 312, 49 ]$ & 	1752 &	432 & 	576 &  	576		& 1.333\\
$[ 312, 52 ]$ & 	1176 &	384 & 	480 &  	480		& 1.250\\
$[ 312, 53 ]$ & 	1192 &	384 & 	480 &  	480		& 1.250\\
$[ 364, 5 ]$ & 		1372 &	448 & 	560 &  	*560		& 1.250\\
$[ 364, 6 ]$ & 		1396 &	448 & 	560 &  	*560    	& 1.250\\
$[ 372, 7 ]$ & 		2172 &	432 & 	504 &  	*\textbf{720}	& 1.667\\
$[ 384, 609 ]$ & 	1740 & 	864 & 	900 &  	\textbf{1008}	& 1.167\\
$[ 384, 618 ]$ & 	2796 &	864 & 	900 &  	\textbf{1008}	& 1.167\\
$[ 390, 1 ]$ & 		2190 &	540 & 	720 &  	720		& 1.333\\
$[ 390, 3 ]$ & 		2214 &	540 & 	720 &  	720		& 1.333\\
$[ 392, 41 ]$ & 	1352 &	448 & 	560 &  	*\textbf{784}	& 1.750\\
$[ 405, 15 ]$ & 	2005 &	675 & 	810 &  	*\textbf{945}	& 1.400\\
$[ 416, 66 ]$ & 	1568 &	512 & 	640 &  	*640		& 1.250\\
$[ 416, 67 ]$ & 	1584 &	512 & 	640 &  	**640		& 1.250\\
$[ 416, 68 ]$ & 	1592 &	512 & 	640 &  	**640		& 1.250\\
$[ 416, 69 ]$ & 	1592 &	512 & 	640 &  	**640		& 1.250\\
$[ 416, 81 ]$ & 	1584 &	512 & 	640 &  	**640		& 1.250\\
$[ 416, 82 ]$ & 	2360 &	512 & 	640 &  	**640		& 1.250\\
$[ 416, 83 ]$ & 	2360 &	512 &	640 &  	**640		& 1.250\\
$[ 416, 84 ]$ & 	2360 &	512 &	640 &  	*640		& 1.250\\
$[ 416, 85 ]$ & 	2360 &	512 &	640 &  	*640		& 1.250\\
$[ 416, 202 ]$ & 	1568 &	512 &	640 &  	*640		& 1.250\\
$[ 416, 203 ]$ & 	1584 &	512 &	640 &  	*640		& 1.250\\
$[ 416, 204 ]$ & 	1584 &	512 &	640 &  	*640		& 1.250\\
$[ 416, 206 ]$ & 	2352 &	512 &	640 &  	*640		& 1.250\\
$[ 416, 208 ]$ & 	2352 &	512 &	640 &  	*640		& 1.250\\
$[ 416, 211 ]$ & 	1584 &	512 &	640 &  	*640		& 1.250\\
$[ 416, 233 ]$ & 	1568 &	512 &	640 &  	*640		& 1.250\\
$[ 468, 7 ]$ & 		5220 &	576 &  	720 &  	*\textbf{864}	& 1.500\\
$[ 468, 9 ]$ & 		1764 &	576 &  	720 &  	*720		& 1.250\\
$[ 468, 10 ]$ & 	1796 &	576 &  	720 &  	*720		& 1.250\\
$[ 468, 31 ]$ & 	4380 &	864 &  	936 &  	\textbf{1080}	& 1.250\\
$[ 468, 33 ]$ & 	2628 &	648 &  	864 &  	*864		& 1.333\\
$[ 468, 35 ]$ & 	2652 &	648 &  	864 &  	*864		& 1.333\\
$[ 468, 36 ]$ & 	1764 &	576 &  	720 &  	*720		& 1.250\\
$[ 468, 37 ]$ & 	1788 &	576 &  	720 &  	*720		& 1.250\\
$[ 468, 38 ]$ & 	1796 &	576 &  	720 &  	*720		& 1.250\\
$[ 520, 38 ]$ & 	3528 &	640 &  	800 &  	*800		& 1.250\\
$[ 520, 39 ]$ & 	1960 &	640 &  	800 &  	*800		& 1.250\\
$[ 520, 44 ]$ & 	1992 &	640 &  	800 &  	**800		& 1.250\\
$[ 546, 1 ]$ & 		3066 &	756 &  	1008 &	**1008		& 1.333\\
$[ 546, 3 ]$ & 		3102 &	756 &  	1008 &	**1008		& 1.333\\
$[ 546, 5 ]$ & 		3138 &	756 &  	1008 &	**1008		& 1.333\\
$[ 546, 6 ]$ & 		3138 &	756 &  	1008 &	**1008		& 1.333\\
$[ 558, 7 ]$ & 		3258 &	648 &  	756 &  	**\textbf{972} 	& 1.500\\
$[ 572, 5 ]$ & 		2156 &	704 &  	880 &  	**880		& 1.250\\
$[ 572, 6 ]$ & 		2196 &	704 &  	880 &  	**880		& 1.250\\
$[ 588, 47 ]$ & 	2316 &	672 &  	840 &  	**\textbf{1176}	& 1.750\\
$[ 588, 48 ]$ & 	2324 &	672 &  	840 &  	\textbf{1344}	& 2.000\\
$[ 588, 49 ]$ & 	2028 &	672 &  	840 &  	\textbf{1008}	& 1.500\\
$[ 600, 162 ]$ & 	3880 &	1280 &	1344 &	\textbf{1536}	& 1.200\\
$[ 602, 2 ]$ & 		4230 &	686 &	784 &  	*\textbf{1078}	& 1.571\\
$[ 620, 7 ]$ & 		6020 &	1000 &	1200 &	\textbf{1600}	& 1.600\\
$[ 624, 8 ]$ & 		3504 &	864 &  	1152 & 	**1152		& 1.333\\
$[ 624, 9 ]$ & 		3528 &	864 &  	1152 & 	**1152		& 1.333\\
$[ 624, 10 ]$ & 	3540 &	864 &  	1152 & 	**1152		& 1.333\\
$[ 624, 11 ]$ & 	3540 &	864 &  	1152 & 	**1152		& 1.333\\
$[ 624, 18 ]$ & 	3528 &	864 &  	1152 & 	*1152		& 1.333\\
$[ 624, 19 ]$ & 	5268 &	864 &  1152 &  	*1152		& 1.333\\
$[ 624, 21 ]$ & 	5268 &	864 &  1152 &  	*1152		& 1.333\\
$[ 624, 23 ]$ & 	3528 &	864 &  1008 &  	\textbf{1152}	& 1.333\\
$[ 624, 117 ]$ & 	6920 &	864 &  1152 &  	*\textbf{1440} 	& 1.667\\
$[ 624, 118 ]$ & 	6984 &	864 &  1152 &  	*\textbf{1440} 	& 1.667\\
$[ 624, 124 ]$ & 	3920 &	768 &  960 &  	*960		& 1.250\\
$[ 624, 126 ]$ & 	3928 &	768 &  960 &  	*960		& 1.250\\
$[ 624, 138 ]$ & 	3504 &	864 &  1152 &  	*1152		& 1.333\\
$[ 624, 139 ]$ & 	3528 &	864 &  1152 &  	*1152		& 1.333\\
$[ 624, 140 ]$ & 	3528 &	864 &  1152 &  	*1152		& 1.333\\
$[ 624, 141 ]$ & 	5256 &	864 &  1152 &  	*1152		& 1.333\\
$[ 624, 142 ]$ & 	5256 &	864 &  1152 &  	*1152		& 1.333\\
$[ 624, 143 ]$ & 	5256 &	864 &  1152 &  	*1152		& 1.333\\
$[ 624, 144 ]$ & 	5256 &	864 &  1152 &  	*1152		& 1.333\\
$[ 624, 146 ]$ & 	3528 &	864 &  1152 &  	*1152		& 1.333\\
$[ 624, 155 ]$ & 	2352 &	768 &  960 &  	*960		& 1.250\\
$[ 624, 156 ]$ & 	2376 &	768 &  960 &  	*960		& 1.250\\
$[ 624, 159 ]$ & 	2376 &	768 &  960 &  	*960		& 1.250\\
$[ 624, 162 ]$ & 	2384 &	768 &  960 &  	*960		& 1.250\\
$[ 624, 163 ]$ & 	2392 &	768 &  960 &  	*960		& 1.250\\
$[ 624, 166 ]$ & 	2392 &	768 &  960 &  	*960		& 1.250\\
$[ 624, 238 ]$ & 	5888 &	1152 &  1296 &	\textbf{1536}	& 1.333\\
$[ 624, 240 ]$ & 	5880 &	1152 & 1296 &	\textbf{1440}	& 1.250\\
$[ 624, 243 ]$ & 	3920 &	1024 & 1152 & 	\textbf{1280}	& 1.250\\
$[ 624, 246 ]$ & 	3504 &	864 & 1152 & 	*1152		& 1.333\\
\hline
\multicolumn{6}{l}{}
\end{tabular}
\begin{tabular}{|lr|rrr|l|} 
\hline
Id($G$)  & $D_3(G)$ &  $\beta_\mathrm{g}(G)$ & $\ell(G)$ & $\ell'(G)$ & $\ell'/\beta_\mathrm{g}$\\
\hline
$[ 624, 249 ]$ & 	2352 &	768 & 960 & 	*960		& 1.250\\
$[ 624, 250 ]$ & 	2384 &	768 & 960 & 	*960		& 1.250\\
$[ 656, 50 ]$ & 	5136 &	1024 & 1088 & 	\textbf{1152}	& 1.125\\
$[ 676, 3 ]$ & 		2692 &	832 & 1040 & 	*1040		& 1.250\\
$[ 676, 9 ]$ & 		2548 &	832 & 1040 & 	*1040		& 1.250\\
$[ 676, 10 ]$ & 	2692 &	832 & 1040 & 	*1040		& 1.250\\
$[ 676, 11 ]$ & 	2692 &	832 & 1040 & 	*1040		& 1.250\\
$[ 676, 12 ]$ & 	2596 &	832 & 1040 & 	*1040		& 1.250\\
$[ 686, 5 ]$ & 		8414 &	1372 & 1568 & 	\textbf{2156}	& 1.571\\
$[ 702, 7 ]$ & 		3942 &	972 & 1296 & 	*1296		& 1.333\\
$[ 702, 8 ]$ & 		4014 &	972 & 1296 & 	*1296		& 1.333\\
$[ 702, 9 ]$ &		4014 &	972 & 1296 & 	*1296		& 1.333\\
$[ 702, 12 ]$ & 	4014 &	972 & 1296 & 	*1296		& 1.333\\
$[ 702, 18 ]$ & 	4134 &	972 & 1296 & 	*1296		& 1.333\\
$[ 702, 19 ]$ & 	4134 &	972 & 1296 & 	*1296		& 1.333\\
$[ 702, 20 ]$ & 	3990 &	972 & 1296 & 	*1296		& 1.333\\
$[ 702, 49 ]$ & 	3942 &	972 & 1296 & 	*1296		& 1.333\\
$[ 702, 51 ]$ & 	3978 &	972 & 1296 & 	*1296		& 1.333\\
$[ 702, 53 ]$ & 	3990 &	972 & 1296 & 	*1296		& 1.333\\
$[ 710, 1 ]$ & 		7010 &	1000 & 1100 & 	*\textbf{1300} 	& 1.300\\
$[ 726, 6 ]$ & 		4326 &	792 & 924 & 	*\textbf{1188}	& 1.500\\
$[ 728, 32 ]$ & 	5096 &	896 & 1120 & 	*1120		& 1.250\\
$[ 728, 34 ]$ & 	2744 &	896 & 1120 & 	*1120		& 1.250\\
$[ 728, 35 ]$ & 	2792 &	896 & 1120 & 	*1120		& 1.250\\
$[ 732, 7 ]$ & 		8652 &	864 & 1008 & 	\textbf{1728}	& 2.000\\
$[ 744, 8 ]$ & 		4344 &	864 & 1008 & 	*\textbf{1296}	& 1.500\\
$[ 744, 9 ]$ & 		4356 &	864 & 1008 & 	*\textbf{1296}	& 1.500\\
$[ 744, 11 ]$ & 	4356 &	864 & 1008 & 	*\textbf{1296}	& 1.500\\
$[ 744, 44 ]$ & 	4344 &	864 & 1008 & 	*\textbf{1296}	& 1.500\\
$[ 750, 28 ]$ & 	3942 &	900 & 1050 & 	*\textbf{1350}	& 1.500\\
$[ 780, 20 ]$ & 	7884 &	1080 & 1440 & 	*1440		& 1.333\\
$[ 780, 21 ]$ & 	4380 &	1080 & 1440 & 	*1440		& 1.333\\
$[ 780, 23 ]$ & 	4428 &	1080 & 1440 & 	*1440		& 1.333\\
$[ 780, 24 ]$ & 	2940 &	960 & 1200 & 	*1200		& 1.250\\
$[ 780, 29 ]$ & 	2988 &	960 & 1200 & 	*1200		& 1.250\\
$[ 780, 30 ]$ & 	2980 &	960 & 1200 & 	*1200		& 1.250\\
$[ 780, 35 ]$ & 	2996 &	960 & 1200 & 	*1200		& 1.250\\
$[ 784, 105 ]$ & 	4632 &	896 & 1120 & 	\textbf{1792}	& 2.000\\
$[ 784, 106 ]$ & 	4632 &	896 & 1120 & 	*\textbf{1344}	& 1.500\\
$[ 784, 110 ]$ & 	6168 &	896 & 1120 & 	\textbf{1792}	& 2.000\\
$[ 784, 113 ]$ & 	3088 &	896 & 1120 & 	\textbf{1792}	& 2.000\\
$[ 784, 114 ]$ & 	3096 &	896 & 1120 & 	\textbf{1792}	& 2.000\\
$[ 784, 117 ]$ & 	3096 &	896 & 1120 & 	\textbf{1344}	& 1.500\\
$[ 784, 124 ]$ & 	2704 &	896 & 1120 & 	\textbf{1344}	& 1.500\\
$[ 784, 125 ]$ & 	2808 &	896 & 1120 & 	\textbf{1344}	& 1.500\\
$[ 784, 126 ]$ & 	2904 &	896 & 1120 & 	\textbf{1344}	& 1.500\\
$[ 832, 181 ]$ & 	3136 &	1024 & 1280 & 	*1280		& 1.250\\
$[ 832, 182 ]$ & 	3168 &	1024 & 1280 & 	*1280		& 1.250\\
$[ 832, 183 ]$ & 	3248 &	1024 & 1280 & 	*1280		& 1.250\\
$[ 832, 184 ]$ & 	3248 &	1024 & 1280 & 	*1280		& 1.250\\
$[ 832, 185 ]$ & 	3256 &	1024 & 1280 & 	*1280		& 1.250\\
$[ 832, 186 ]$ & 	2156 &	1024 & 1280 & 	*1280		& 1.250\\
$[ 832, 187 ]$ & 	3192 &	1024 & 1280 & 	*1280		& 1.250\\
$[ 832, 188 ]$ & 	3192 &	1024 & 1280 & 	*1280		& 1.250\\
$[ 832, 200 ]$ & 	3184 &	1024 & 1280 & 	*1280		& 1.250\\
$[ 832, 201 ]$ & 	3248 &	1024 & 1280 & 	*1280		& 1.250\\
$[ 832, 203 ]$ & 	4824 &	1024 & 1280 & 	*1280		& 1.250\\
$[ 832, 206 ]$ & 	4760 &	1024 & 1280 & 	*1280		& 1.250\\
$[ 832, 207 ]$ & 	4760 &	1024 & 1280 & 	*1280		& 1.250\\
$[ 832, 212 ]$ & 	4720 &	1024 & 1280 & 	*1280		& 1.250\\
$[ 832, 213 ]$ & 	4720 &	1024 & 1280 & 	*1280		& 1.250\\
$[ 832, 230 ]$ & 	3168 &	1024 & 1280 & 	*1280		& 1.250\\
$[ 832, 231 ]$ & 	3184 &	1024 & 1280 & 	*1280		& 1.250\\
$[ 832, 232 ]$ & 	3248 &	1024 & 1280 & 	*1280		& 1.250\\
$[ 832, 233 ]$ & 	3248 &	1024 & 1280 & 	*1280		& 1.250\\
$[ 832, 237 ]$ & 	4784 &	1024 & 1280 & 	*1280		& 1.250\\
$[ 832, 238 ]$ & 	4720 &	1024 & 1280 &	*1280		& 1.250\\
$[ 832, 240 ]$ & 	4784 &	1024 & 1280 & 	*1280		& 1.250\\
$[ 832, 241 ]$ & 	5560 & 	1024 & 1280 & 	*1280		& 1.250\\
$[ 832, 242 ]$ & 	5496 &	1024 & 1280 & 	*1280		& 1.250\\
$[ 832, 245 ]$ & 	5560 &	1024 & 1280 & 	*1280		& 1.250\\
$[ 832, 246 ]$ & 	5496 &	1024 & 1280 & 	*1280		& 1.250\\
$[ 832, 254 ]$ & 	3288 &	1024 & 1280 & 	*1280		& 1.250\\
$[ 832, 258 ]$ & 	4760 &	1024 & 1280 & 	*1280		& 1.250\\
$[ 832, 264 ]$ & 	4760 &	1024 & 1280 & 	*1280		& 1.250\\
$[ 858, 1 ]$ & 		4818 &	1188 & 1584 & 	*1584		& 1.333\\
$[ 858, 3 ]$ & 		4878 &	1188 & 1584 & 	*1584		& 1.333\\
$[ 876, 7 ]$ & 		10380 &	1728 & 1872 & 	\textbf{2880} 	& 1.667\\
$[ 884, 6 ]$ & 		3332 &	1088 & 1360 & 	*1360		& 1.250\\
$[ 884, 7 ]$ & 		3524 &	1088 & 1360 & 	*1360		& 1.250\\
$[ 884, 8 ]$ & 		3524 &	1088 & 1360 & 	*1360		& 1.250\\
$[ 884, 9 ]$ & 		3396 &	1088 & 1360 & 	*1360		& 1.250\\
$[ 888, 45 ]$ & 	10392 &	1728 & 1800 & 	\textbf{1872}	& 1.083\\
$[ 903, 1 ]$ & 		18543 &	1323 & 1764 & 	\textbf{2646}	& 2.000\\
$[ 903, 2 ]$ & 		6195 &	1029 & 1176 & 	*\textbf{1470}	& 1.428\\
$[ 915, 1 ]$ & 		13515 &	1125 & 1350 & 	*\textbf{1800}	& 1.600\\
$[ 930, 3 ]$ & 		9030 &	1500 & 1800 & 	\textbf{2400}	& 1.600\\
$[ 930, 5 ]$ & 		9050 &	1500 & 1800 & 	\textbf{2400}	& 1.600\\
$[ 930, 6 ]$ & 		5430 &	1080 & 1260 & 	*\textbf{1440}	& 1.333\\
$[ 930, 8 ]$ & 		5454 &	1080 & 1260 & 	*\textbf{1440}	& 1.333\\
$[ 968, 35 ]$ & 	7688 &	1408 & 1584 & 	\textbf{2816}	& 2.000\\
$[ 968, 36 ]$ & 	6412 &	1408 & 1584 & 	*\textbf{1936}	& 1.222\\
$[ 968, 37 ]$ & 	7692 &	1408 & 1584 & 	\textbf{2816}	& 2.000\\
$[ 979, 1 ]$ & 		10659 &	1331 & 1452 & 	*\textbf{1573}	& 1.181\\
$[ 980, 18 ]$ & 	3860 &	1120 & 1400 & 	*\textbf{1680}	& 1.500\\
$[ 980, 23 ]$ & 	3876 &	1120 & 1400 & 	\textbf{2240}	& 2.000\\
$[ 980, 24 ]$ & 	3380 &	1120 & 1400 & 	\textbf{1680}	& 1.500\\
$[ 980, 27 ]$ & 	3588 &	1120 & 1400 & 	\textbf{1680}	& 1.500\\
$[ 980, 28 ]$ & 	3780 &	1120 & 1400 & 	\textbf{1680}	& 1.500\\
$[ 984, 33 ]$ & 	7704 &	1536 & 1600 & 	\textbf{1728}	& 1.125\\
$[ 984, 34 ]$ & 	7720 &	1536 & 1600 & 	\textbf{1728}	& 1.125\\
$[ 988, 5 ]$ & 		3724 &	1216 & 1520 & 	*1520		& 1.125\\
$[ 988, 6 ]$ & 		3796 &	1216 & 1520 & 	*1520		& 1.125\\
$[ 994, 1 ]$ & 		13734 &	1372 & 1568 & 	*\textbf{2156}	& 1.571\\\hline
\end{tabular}
\normalsize
\caption{Results of the upgrade process for subsets with input from Table \ref{tab:BiggGrp}.}
\label{tab:BiggSet}
\end{table}

\begin{table}
\footnotesize
\begin{tabular}{|lrr|rrr|l|}
\hline
Id($G$) & $D_3(G)$ & $\beta_\mathrm{g}(G)$ & $\ell_\text{\textsc{RandomDelete}}$ & $\ell_\text{\textsc{MaxTripleDelete}}$ & $\ell_\text{\textsc{MaxQuotientDelete}}$ & $\ell / \beta_\mathrm{g}$\\
\hline
$[ 168, 1 ]$ 	& 888  & 216  & \textbf{288} 	& 216 		& 216 		& 1.333\\
$[ 168, 10 ]$ 	& 888  & 216  & \textbf{288} 	& 216 		& 216 		& 1.333\\
$[ 171, 3 ]$ 	& 1467 & 243  & 243 		& \textbf{324} 	& 243		& 1.333\\
$[ 192, 185 ]$ 	& 1856 & 384  & 384 		& \textbf{480} 	& 384		& 1.250\\
$[ 196, 8 ]$ 	& 772  & 224  & \textbf{240} 	& \textbf{280} 	& \textbf{240}	& 1.250\\
$[ 252, 18 ]$ 	& 1356 & 324  & 324 		& \textbf{432} 	& 324		& 1.333\\
$[ 256, 5713 ]$ & 1392 & 512  & 512 		& \textbf{640} 	& 512		& 1.250\\
$[ 256, 5715 ]$ & 1392 & 512  & 512 		& \textbf{640} 	& 512		& 1.250\\
$[ 336, 121 ]$ 	& 1800 & 432  & \textbf{576} 	& 432 		& 432		& 1.333\\
$[ 342, 8 ]$ 	& 2934 & 486  & 486 		& \textbf{648} 	& 486		& 1.333\\
$[ 384, 576 ]$ 	& 2880 & 768  & 768 		& \textbf{960} 	& 768		& 1.250\\
$[ 384, 617 ]$ 	& 2412 & 864  & 864 		& \textbf{1008} & 864		& 1.167\\
$[ 392, 40 ]$ 	& 1544 & 448  & \textbf{464} 	& \textbf{560} 	& \textbf{464}	& 1.250\\
$[ 406, 1 ]$ 	& 5502 & 686  & \textbf{784} 	& 686 		& 686		& 1.142\\
$[ 420, 3 ]$ 	& 2268 & 540  & 540 		& \textbf{720} 	& 540		& 1.333\\
$[ 420, 15 ]$ 	& 4092 & 864  & \textbf{936} 	& 864 		& 864		& 1.083\\
$[ 441, 9 ]$ 	& 3249 & 729  & \textbf{810} 	& \textbf{810} 	& \textbf{810}	& 1.111\\
$[ 486, 19 ]$ 	& 2370 & 972  & 972 		& \textbf{1134} & 972		& 1.167\\
$[ 486, 176 ]$ 	& 6258 & 972  & 972 		& \textbf{1296} & 972		& 1.333\\
$[ 486, 178 ]$ 	& 6258 & 972  & 972 		& \textbf{1296} & 972		& 1.333\\
$[ 500, 17 ]$ 	& 8340 & 1000 & \textbf{1200} 	& \textbf{1200} & 1000		& 1.200\\
$[ 500, 21 ]$ 	& 8356 & 1000 & \textbf{1200} 	& \textbf{1200} & 1000		& 1.200\\
$[ 504, 44 ]$ 	& 2712 & 648  & 648 		& \textbf{864} 	& 648		& 1.333\\
$[ 504, 68 ]$ 	& 4440 & 864  & \textbf{1008} 	& 864 		& 864		& 1.167\\
$[ 504, 85 ]$ 	& 2724 & 648  & 648 		& \textbf{864} 	& 648		& 1.333\\
$[ 504, 88 ]$ 	& 2712 & 648  & 648 		& \textbf{864} 	& 648		& 1.333\\
$[ 513, 13 ]$ 	& 4401 & 729  & \textbf{972} 	& \textbf{972} 	& 729		& 1.333\\
$[ 576, 183 ]$ 	& 3924 & 768  & \textbf{864} 	& \textbf{960} 	& \textbf{800}	& 1.250\\
$[ 576, 185 ]$ 	& 2984 & 768  & \textbf{800} 	& 768 		& \textbf{800}	& 1.041\\
$[ 576, 188 ]$ 	& 5792 & 1152 & 1152 		& \textbf{1440} & 1152		& 1.250\\
$[ 588, 16 ]$ 	& 2828 & 756  & 756 		& \textbf{1008} & 756		& 1.333\\
$[ 588, 19 ]$ 	& 3180 & 756  & 756 		& \textbf{1008} & 756		& 1.333\\
$[ 588, 21 ]$ 	& 3252 & 756  & \textbf{1008} 	& 756 		& 756		& 1.333\\
$[ 588, 22 ]$ 	& 3468 & 756  & \textbf{1008} 	& 756 		& 756		& 1.333\\
$[ 657, 3 ]$ 	& 5841 & 729  & \textbf{810} 	& \textbf{810} 	& \textbf{810}	& 1.111\\
$[ 672, 37 ]$ 	& 6228 & 864  & \textbf{1152} 	& 864 		& 864		& 1.333\\
$[ 672, 39 ]$ 	& 6228 & 864  & \textbf{1152} 	& 864 		& 864		& 1.333\\
$[ 684, 17 ]$ 	& 5868 & 972  & 972 		& \textbf{1296} & 972		& 1.333\\
$[ 684, 30 ]$ 	& 5868 & 972  & 972 		& \textbf{1296} & 972		& 1.333\\
$[ 684, 31 ]$ 	& 5922 & 972  & 972 		& \textbf{1296} & 972		& 1.333\\
$[ 737, 1 ]$ 	& 7997 & 1331 & \textbf{1452} 	& \textbf{1452} & \textbf{1452}	& 1.090\\
$[ 756, 27 ]$ 	& 4380 & 972 & 972 & \textbf{1296} & 972 & 1.333\\
$[ 756, 28 ]$ 	& 4380 & 972 & 972 & \textbf{1296} & 972 & 1.333\\
$[ 756, 29 ]$ 	& 4092 & 972 & 972 & \textbf{1296} & 972 & 1.333\\
$[ 756, 31 ]$ 	& 4380 & 972 & 972 & \textbf{1296} & 972 & 1.333\\
$[ 756, 138 ]$ 	& 4068 & 972 & 972 & \textbf{1296} & 972 & 1.333\\
$[ 756, 140 ]$ 	& 4092 & 972 & 972 & \textbf{1296} & 972 & 1.333\\
$[ 756, 151 ]$ 	& 7740 & 1728 & \textbf{1872} & 1728 & \textbf{1872} & 1.083\\
$[ 812, 1 ]$ 	& 11004& 1372 & \textbf{1568} & 1372 & 1372 & 1.142\\
$[ 840, 3 ]$ 	& 4536 & 1080 & 1080 & \textbf{1440} & 1080 & 1.333\\
$[ 840, 15 ]$ 	& 8184 & 1080 & 1080 & \textbf{1440} & 1080 & 1.333\\
$[ 840, 33 ]$ 	& 4548 & 1080 & 1080 & \textbf{1440} & 1080 & 1.333\\
$[ 840, 36 ]$ 	& 4536 & 1080 & 1080 & \textbf{1440} & 1080 & 1.333\\
$[ 855, 3 ]$ 	& 7335 & 1215 & 1215 & \textbf{1620} & 1215 & 1.333\\
$[ 882, 37 ]$ 	& 6498 & 1458 & \textbf{1620} & \textbf{1620} & \textbf{1539} & 1.111\\
$[ 924, 3 ]$ 	& 5004 & 1188 & 1188 & \textbf{1584} & 1188 & 1.333\\\hline
\multicolumn{3}{|r|}{\# upgrades} & 23 & 43 & 9 &\\
\multicolumn{6}{|r|}{max.} & 1.333\\\hline
\end{tabular}
\caption{Results of two upgrade steps (process for subsets) started after one reduction step with input from \cite{HedtkeMurthy2011}.}
\label{tab:Red}
\end{table}

\providecommand{\bysame}{\leavevmode\hbox to3em{\hrulefill}\thinspace}
\providecommand{\MR}{\relax\ifhmode\unskip\space\fi MR }
\providecommand{\MRhref}[2]{%
  \href{http://www.ams.org/mathscinet-getitem?mr=#1}{#2}
}
\providecommand{\href}[2]{#2}

\end{document}